\documentclass[12pt, leqno]{article}
\usepackage{amsmath, amsthm, enumitem, thmtools}
\usepackage{amsfonts}
\usepackage{amssymb}
\usepackage[usenames,dvipsnames]{color}%\usepackage[dvips]{color}
\usepackage{graphicx}
\usepackage{bbm}
\usepackage[toc]{appendix}
\usepackage{bm}
\usepackage{listings}
\usepackage{nicefrac} % For comparison
\usepackage{xfrac}    % Works better with other fonts
\usepackage[utf8]{inputenc}
\usepackage[T1]{fontenc}
\usepackage{hyperref}
\usepackage{pgffor}
\usepackage{xcolor}
\usepackage{soul}

\usepackage{color,calc}

\makeatletter
\definecolor{shade}{gray}{0.8}
        {
          \raggedright
        \setlength{\rightmargin}{\leftmargin}
        \setlength{\itemsep}{-12pt}
        \setlength{\parsep}{20pt}
        \begin{lrbox}{\@tempboxa}%
        \begin{minipage}{\linewidth-2\fboxsep}
        }%
        {
        \end{minipage}%
        \end{lrbox}%
        \fcolorbox{black}{shade}{\usebox{\@tempboxa}}\newline
        }%
\makeatother

\newtheorem{theorem}{Theorem}
\newtheorem{lemma}{Lemma}

\renewcommand{\d}{{\rm d}}

\renewcommand{\eqref}[1]{\hyperref[#1]{(\ref*{#1})}}

\newcommand{\dd}{\mathrm{d}}

\newcommand*{\norm}[1]{\lVert #1 \rVert}

\newtheorem{remark}{Remark}
\newtheorem{definition}{Definition}

\renewcommand{\ln}{\log}

\newcommand*{\pref}[1]{\hyperref[#1]{(\ref*{#1})}}
\newcommand*{\refpref}[2]{\hyperref[#2]{\ref*{#1}(\ref*{#2})}}

\newcommand{\N}{\mathbb{N}}%Naturales
\newcommand{\sT}{\mathtt{T}}
\newcommand{\sP}{\mathtt{P}}

\newcommand{\sN}{\mathtt{N}}

\newcommand{\sG}{\mathtt{G}}
\newcommand{\sV}{\mathtt{V}}
\newcommand{\sW}{\mathtt{W}}

\newcommand{\sv}{\mathtt{v}}
\newcommand{\su}{\mathtt{u}}
\newcommand{\sA}{\mathtt{A}}

\newcommand{\sm}{\mathtt{m}}

\newcommand{\sF}{\mathtt{F}}

\newcommand{\bP}{\mathtt{P}}
\newcommand{\sJ}{\mathtt{J}}

\newcommand{\V}{\mathbb{V}}

\definecolor{amethyst}{rgb}{0.6, 0.4, 0.8}
\definecolor{applegreen}{rgb}{0.55, 0.71, 0.0}
\definecolor{aqua}{rgb}{0.0, 1.0, 1.0}
\definecolor{asparagus}{rgb}{0.53, 0.66, 0.42}
\definecolor{amber(sae/ece)}{rgb}{1.0, 0.49, 0.0}
 	\definecolor{armygreen}{rgb}{0.29, 0.33, 0.13}
	\definecolor{shitbrown}{rgb}{0.43, 0.21, 0.1}
	\definecolor{brightpink}{rgb}{1.0, 0.0, 0.5}
	\definecolor{brightube}{rgb}{0.82, 0.62, 0.91}
	 	\definecolor{byzantine}{rgb}{0.74, 0.2, 0.64}
		\definecolor{chartreuse(web)}{rgb}{0.5, 1.0, 0.0}

\title{Stability of (sub)critical non-local spatial  branching processes with and without immigration}

\author{ Emma Horton\thanks{
Department of Statistics
University of Warwick
Coventry
CV4 7AL, UK. E-mail: \texttt{ \{emma.horton\}, \{andreas.kyprianou\} @warwick.ac.uk}
}, \ Andreas E. Kyprianou$^*$, \ Pedro Mart\'in-Ch\'avez\thanks{Department of Mathematics, University of Extremadura, Avenida de Elvas s/n 06006 Badajoz, Spain. E-mail: \texttt{pedromc@unex.es}. ORCID: \href{https://orcid.org/0000-0001-5530-3138}{0000-0001-5530-3138}}, \\ Ellen Powell\thanks{Department of Mathematical Sciences, Durham University Upper Mountjoy Campus, Stockton Rd, Durham DH1 3LE. Email \texttt{ellen.g.powell@durham.ac.uk}}, \ Victor Rivero\thanks{CIMAT A. C., Calle Jalisco s/n, Col. Valenciana, A. P. 402, C.P. 36000, Guanajuato, Gto., Mexico. Email \texttt{rivero@cimat.mx}}
}

\newcommand{\bra}[1]{\ensuremath{\left[ #1\right] }}
\newcommand{\proint}[2]{\ensuremath{\left\langle #1,#2\right\rangle}}

\begin{document}

\maketitle
\begin{abstract}%\footnote{\azul Pedro: rewrite the abstract a bit to include the result of the gamma limit distribution}
\noindent {We consider the setting of either a general non-local branching particle process or a general non-local superprocess, in both cases, with and without immigration. Under the assumption that the mean semigroup has a Perron-Frobenious type behaviour for the immigrated mass, as well as the existence of second moments,  we consider necessary and sufficient conditions that ensure limiting distributional stability.  More precisely, our first main contribution pertains to proving the asymptotic Kolmogorov survival probability and Yaglom limit for critical non-local branching particle systems {\it and} superprocesses under a second moment assumption on the offspring distribution. Our results improve on existing literature by removing the requirement of bounded offspring in the particle setting \cite{HHKW} and generalising \cite{RSZ} to allow for non-local branching mechanisms. Our second main contribution pertains to the stability of both critical and sub-critical non-local branching particle systems and superprocesses with immigration. At criticality, we show that the scaled process converges to a Gamma distribution under a necessary and sufficient integral test. At subcriticality we show stability of the process, also subject to an integral test.
In these cases, our results complement classical results for (continuous-time) Galton--Watson processes with immigration and continuous-state branching processes with immigration; see \cite{Heathcote1965, Pinsky, Quine, Seva1957, Yang}, among others. In the setting of superprocesses, the only work we know of at this level of generality is summarised in \cite{ZL11}.
The proofs of our results, both with and without immigration, appeal to similar technical approaches and accordingly, we include the results together in this paper. 
}
%Our proofs consolidate recent work for the setting of non-local spatial branching processes which have established analogues of the classical Yaglom limit for the Galton--Watson setting, see \cite{HHKW}, as well as providing growth rates for moments, see \cite{GHK, GHKcorr}. 

\medskip

\medskip

\end{abstract}
%\tableofcontents

\section{Introduction}
In this article, we revisit foundational results concerning the stability of branching processes with and without immigration. 
In essence, our objective is to show that, qualitatively speaking, several of the classical results for Galton--Watson processes (with and without immigration) are universal truths in the setting of general branching Markov processes and superprocessses. 

In what follows, we focus on critical or subcritical processes. In the setting of Galton--Watson processes, the notion of criticality is dictated by the mean number of offspring. In the general setting we present in this work, the notion of criticality pertains to the value of an assumed lead eigenvalue for the mean semigroup. 

The first main focus of this paper pertains to critical processes without immigration. In this case, we are interested in the extent to which the so-called Kolmogorov and Yaglom limits for discrete-time Galton--Watson processes are still an inherent behaviour at generality. The Kolmogorov limit stipulates that the decay of the survival probability at criticality is inversely proportional  to time (interpreted  as either real-time or generation number). The Yaglom limit asserts that, conditional on survival, the current population normalised by time, converges to an exponential random variable with rate that is written explicitly in terms of the model parameters. In both the Kolmogorov and Yaglom limits, second moments of the offspring distribution are needed in the classical setting. 

Improving on recent work in this domain for general non-local branching Markov processes \cite{HHKW} and  superprocesses \cite{RSZ}, we prove both of these results under a second moment assumption on the offspring distribution. For non-local branching Markov processes, this builds on \cite{HHKW}, where a bound on the  number of offspring was required. In the setting of superprocesses, we accommodate for non-local branching mechanisms, where previous works have only allowed local branching \cite{RSZ}.

The third and fourth main results of this article concern critical and subcritical processes {\it with} immigration. Returning to the Galton--Watson setting, let us consider the case where we have i.i.d. immigration in each generation, with each immigrant spawning an independent copy of the underlying Galton--Watson process. If $f(s) = \mathcal{E}[s^N]$, $s\in[0,1]$ is the probability generating function of the offspring distribution of the typical family size, $N$, for the Galton--Watson dynamics and $g(s) = \Tilde{\mathcal{E}}[s^{\Tilde{N}}]$, $s\in[0,1]$, is the probability generating function of the distribution of the number of immigrants, $\tilde{N}$, in each generation,  then it is known (see \cite{AN, FW, Heathcote1965, Seneta1968, Pakes1971}) that the total population converges in distribution if, and only if, the process is not supercritical, i.e. $\mathcal{E}[N]   \leq 1$, and
\begin{equation}
\int_0^1 \frac{1-g(s)}{f(s)-s}\dd s <\infty.
\label{classicalintegral}
\end{equation}
In the subcritical setting, the integral \eqref{classicalintegral} is equivalent to the requirement that $\tilde{\mathcal{E}}[\log(1+\tilde{N})]<\infty$. In the critical setting, although the integral \eqref{classicalintegral} fails, it is still possible to demonstrate that the process with immigration when scaled by time converges to a gamma distribution {(see \cite{thesisFOSTER,FW,Pakes-1971,Seneta1970})}.

Again, we develop analogous results in the general framework of non-local branching Markov processes and non-local superprocesses with immigration.
In the former case, we believe our results to be the first of their kind at this level of generality. For non-local superprocesses, our results complement those in Chapter 9 of \cite{ZL11}. Indeed, in the setting of independent immigration, at subcriticality, we introduce an integral test which seems not to have been noted previously. At criticality we are able to provide the natural analogue of scaled convergence of the population  to a gamma distribution, which also appears to be new for superprocesses. As in the first two results, we also work under a second moment assumption on the offspring distribution.

It turns out that there is a natural reason to consider the results with and without immigration together. Indeed, a fundamental feature of the analysis in both cases pertains to how the asymptotic behaviour of the non-linear semigroup of the underlying branching process behaves in relation to its linear semigroup. In particular, the insistence of second moments throughout leads to the use of a {\color{black} second order} Taylor approximation in all cases. 

\section{Non-local spatial branching processes}
Let us spend some time describing the general setting in which we wish to work.
Let $E$ be a Lusin space. Throughout, will write $B(E)$ for the Banach space of bounded measurable functions on $E$ with norm $\norm{\cdot}$, $B^{+}(E)$ for the space of non-negative bounded measurable functions on $E$ and $B^{+}_1(E)$ for the subset of functions in $B^{+}(E)$ that are uniformly bounded by unity. We are interested in spatial branching processes that are defined in terms of a Markov process and a branching {\color{black}mechanism}, whether that be a branching particle system or a superprocess. We characterise Markov processes by a semigroup on $E$, denoted by $\sP=(\sP_t, t\geq0)$. Unless otherwise stated, we do not need $\bP$ to have the Feller property, and it is not necessary that $\bP$ is conservative. Indeed, in the case where it is non-conservative, we can append  a cemetery state $\{\dagger\}$ to $ E$, which is to be treated as an absorbing state, and regard $\bP$ as conservative on the extended space $E\cup\{\dagger\}$, which can also be treated as a Lusin space.
However, we must then alter the definition of $B(E)$ (and accordingly $B^+(E)$ and $B^+_1(E)$) to ensure that any function $f\in B(E)$ satisfies $f(\dagger)=0$.

\subsection{Non-local Branching Markov Processes}%\footnote{\azul Pedro: For me in this subsection 2.1 everything is correct and does not need changes. If you finally want to change the nonlinear semigroup $\sv_t[\cdot]$ and use the more common definition with the product, I would also agree and could introduce the corresponding changes.}
Consider now a spatial branching process in which, given their point of creation, particles evolve independently according to a $\sP$-Markov process. In an event, which we refer to as `branching', particles positioned at $x$ die at rate {\color{black}$\beta(x)$, where }$\beta\in B^+(E)$, and instantaneously, new particles are created in $E$ according to a point process. The configurations of these offspring are described by the random counting measure
\[
\mathcal{Z}(A) = \sum_{i = 1}^N \delta_{x_i}( A), 
\]
for Borel subsets $A$ of $E$, {\color{black} for which we also assume that $\sup_{x\in E}\mathcal{E}_x[N]<\infty$.} The law of the aforementioned point process depends on $x$, the point of death of the parent, and we denote it by $\mathcal{P}_x$, $x\in E$, with associated expectation operator given by $\mathcal{E}_x$, $x\in E$.  This information is captured in the so-called branching mechanism
\begin{equation*}
  \sG[g](x) :=  \beta(x)\mathcal{E}_x\left[\prod_{i = 1}^N g(x_i) - g(x)\right], \qquad x\in E,
  \label{linearG}
\end{equation*}
where $ g\in B^+_1(E)$. 
Without loss of generality, we can assume that $\mathcal{P}_x(N =1) = 0$ for all $x\in E$ by viewing a branching event with one offspring as an extra jump in the motion. On the other hand, we do allow for the possibility that $\mathcal{P}_x(N =0)>0$ for some or all $x\in E$. 
\medskip

Henceforth we refer to this spatial branching process as a $(\sP, \sG)$-branching Markov process (or $(\sP, \sG)$-BMP for short). It is well known that if the configuration of particles at time $t$ is denoted by $\{x_1(t), \ldots, x_{N_t}(t)\}$, then, on the event that the process has not become extinct or exploded, the branching Markov process can be described as the co-ordinate process $X= (X_t, t\geq0)$, given by
\[
X_t (\cdot) = \sum_{i =1}^{N_t}\delta_{x_i(t)}(\cdot), \qquad t\geq0,
\]
evolving in the space of {\color{black} counting} measures on $E$ with {\color{black} finite} total mass, which we denote $N(E)$.
In particular, $X$
is Markovian in  $N(E)$. Its probabilities will be denoted $\mathbb{P}: = (\mathbb{P}_\mu, \mu\in N(E))$. 
Sometimes we will write $X^{(\mu)}$ to signify that we are considering $X$ under $\mathbb{P}_\mu$, that is to say, $X^{(\mu)}_0 =\mu$.
For convenience, we will write for any measure $\mu \in N(E)$ and function $f\in B^+(E)$,
\[
\langle f, \mu\rangle = \int_E f(x)\mu(\dd x).
\]
In particular, 
\[
\langle f, X_t\rangle= \sum_{i = 1}^{N_t} f(x_i(t)), \qquad f\in B^+(E).
\]

With this notation in hand, it is worth noting that the independence that is inherent in the definition of the Markov branching property  implies that, if we define, 
\begin{equation*}
{\rm e}^{-\sv_t[f](x)} = \mathbb{E}_{\delta_x}\left[{\rm e}^{-\langle f, X_t\rangle}\right], \qquad t\geq 0, \, f\in B^+(E),\, x\in E,
\label{nonlin}
\end{equation*}
then for $\mu\in N(E)$, %given by $\mu = \sum_{i =1}^n\delta_{y_i}$, 
we have
\begin{equation}
\label{MBP}
\mathbb{E}_{\mu}\left[{\rm e}^{-\langle f, X_t\rangle}\right] = {\rm e}^{-\langle\sv_t[f], \mu\rangle }, \qquad t\geq 0.
\end{equation}
Moreover, for $f\in B^+(E)$ and  $x\in E$, 
\begin{equation}
{\rm e}^{-\sv_t[f](x)} = \sP_t[{\rm e}^{-f}](x) + \int_0^t \sP_s\left[ \sG[{\rm e}^{-\sv_{t-s}[f]}]\right](x)\d s, \qquad t\geq0.
\label{nonlinv}
\end{equation}
 The above equation describes the evolution of the semigroup $\sv_t[\cdot]$ in terms of the action of  transport and branching. That is, either the initial particle has not branched {\color{black}and undergone a Markov transition (including the possibility of being  absorbed)} by time $t$ or at some time $s \le t$, the initial particle has branched, producing offspring according to $\sG$. We refer the reader to~\cite{SNTE-I, SNTE-II} for a proof.

Branching Markov processes enjoy a very long history in the literature, dating back as far as the late 1950s, \cite{Seva1958, Seva1961, Sk1964, INW1, INW2, INW3}, with a broad base of literature that is arguably too voluminous to give a fair summary of here. Most literature focuses on the setting of local branching. This corresponds to the setting that all offspring are positioned at their parent's point of death (i.e. $x_i = x$ in the definition of $\sG$). In that case, the branching mechanism reduces to 
\[
  \sG[s](x) = \beta(x)\left[\sum_{k =0}^\infty p_{k}(x)s^k - s\right], \qquad x\in E,
\]
where $s\in[0,1]$ and $(p_k(x), k\geq 0)$ is the offspring distribution when a parent branches at site $x\in E$.  
The branching mechanism  $\sG$ may otherwise be seen, in general, as a mixture of  local and non-local branching.

We want to introduce a variant of the model that includes immigration, where the new particles can arrive into the system from an external source. These arrival times, at which immigration events occur, are determined by a homogeneous Poisson process with rate $\alpha$. At each arrival time, a random number of particles, $\Tilde{N}$, is added to the system at locations $y_1,\ldots,y_{\Tilde{N}}$ in $E$. {\color{black}The latter} can be summarised by another random counting measure 
\begin{equation}\label{immigr-counting-measure}
    \Tilde{\mathcal{Z}}(\cdot) = \sum_{i=1}^{\Tilde{N}} \delta_{y_i} (\cdot).
\end{equation}
The corresponding law, $\Tilde{\mathcal{P}}$, is independent of the state of the system and its expectation is denoted by $\Tilde{\mathcal{E}}$. Similarly to before, this can be succinctly described by the immigration mechanism
%\begin{equation}
%    \mathtt{H}(f) = \alpha \Tilde{\mathcal{E}} \left[1-\exp\left(-\sum_{i=1}^{\Tilde{N}} f(y_i)\right)\right], \qquad f\in B^+(E),
%\end{equation}
\begin{equation*}
    \mathtt{H}[f] = \alpha \Tilde{\mathcal{E}} \left[1- {\rm e}^{-\langle f, \tilde{\mathcal Z}\rangle}\right], \qquad f\in B^+(E),
\end{equation*}
where we assume $\Tilde{\mathcal{P}}(\Tilde{N}=0)=0$, i.e. we always have at least one immigrant at the arrival times. 

Once immigrants are embedded in the system, they evolve according to the same rules as independent copies of the branching Markov process, initiated from their point of arrival.

\begin{definition}[Non-local branching Markov process with immigration]\label{def-BMPI}
We say that  $Y^{(\mu)}=(Y^{(\mu)}_t, t\geq0)$ is a $(\mathtt{P},\mathtt{G})$-branching Markov process with $\mathtt{H}$-immigration (or a $(\mathtt{P},\mathtt{G},\mathtt{H})$-BMPI for short) with initial mass $\mu\in N(E)$, if 
\begin{equation}
Y^{(\mu)}_t = X^{(\mu)}_t  + \sum_{j=1}^{D_t} X_{t-\tau_j}^{(\Tilde{\mathcal{Z}}_j)}, \qquad t\geq0,
\label{BMPI}
\end{equation}
where $(D_t,t\geq 0)$ is the homogeneous Poisson process with rate $\alpha$, $\tau_j$ is $j$-th arrival time and $\{\Tilde{\mathcal{Z}}_j, j\in\mathbb{N}\}$ are i.i.d. copies of $\Tilde{\mathcal{Z}}$. Moreover, given $(\Tilde{\mathcal{Z}}_j, j = 1, \cdots, D_t)$, the processes $ X^{(\Tilde{\mathcal{Z}}_j)}$ are independent copies of $X^{(\mu)}$ issued from the respective measures $\mu = \Tilde{\mathcal{Z}}_j$.
The probabilities of $Y^{(\mu)}$ are also denoted by $\mathbb{P}_{\mu}$, $\mu\in N(E)$.
\end{definition}

\subsection{Non-local Superprocesses} 
%\footnote{\azul Pedro: For me in this subsection 2.2 everything is correct and does not need changes}
Superprocesses can be thought of as the high-density limit of a sequence of branching Markov processes, resulting in a new family of measure-valued Markov processes; see e.g.  \cite{ZL11, Dawson1993, Wat1968, Dynkin2, dawson2002nonlocal}. Just as branching Markov processes are Markovian in $N(E)$,  the former are Markovian in $M(E)$, the space of finite Borel measures on $E$ equipped with the {\color{black} topology of weak convergence}.  
There is a broad literature  for superprocesses, e.g. \cite{ZL11, Dawson1993, Wat1968, Alison, Janos}, with so-called local branching mechanisms, which has been  broadened to the more general setting of non-local branching mechanisms in \cite{dawson2002nonlocal, ZL11}. Let us now introduce these concepts with a self-contained  definition of what we mean by a non-local superprocess (although the reader will note that we largely conform to the presentation in \cite{ZL11}).

\medskip

A Markov process $X : = (X_t:t\geq 0)$ with state space $M(E)$  and probabilities $\mathbb{P} := (\mathbb{P}_\mu, \mu\in M(E))$ is called a $(\sP,\psi,\phi)$-superprocess  (or $(\sP,\psi,\phi)$-SP for short) if it has  semigroup $(\sV_t, t\geq 0)$ on $M(E)$ satisfying 
\begin{equation}\label{non-local-transition-semigroup}
\mathbb{E}_\mu\big[{\rm e}^{-\langle{f},{X_t}\rangle}\big] ={\rm e}^{-\langle{\sV}_t[f], \mu\rangle}, \qquad \mu\in M(E), f\in B^{+}(E),
\end{equation}
where  $({\sV}_t, t\geq 0)$ is   characterised as the {\color{black} minimal non-negative   solution} of the evolution equation
\begin{equation}\label{non-local-evolution-equation}
{\sV}_t[f](x)=\sP_t[f](x)-\int_{0}^{t}\sP_s\big[\psi(\cdot,{\sV}_{t-s}[f](\cdot))+\phi(\cdot,{\sV}_{t-s}[f])\big](x)\d s.
\end{equation}
Here $\psi$ denotes the local branching mechanism
\begin{equation}\label{local-branching-mechanism}
\psi(x,\lambda)=-b(x) \lambda + c(x)\lambda^2 +\int_{0}^\infty( {\rm e}^{-\lambda y}-1+\lambda y)\nu(x,\d y),\;\;\;\lambda\geq 0,
\end{equation}
where $b\in B(E)$, $c\in B^+(E)$ and $(x\wedge x^2)\nu(x, \d y)$ is a {\color{black} uniformly (for $x\in E$)} bounded kernel from $E$ to $(0,\infty)$,
and $\phi$ is the non-local branching mechanism
\begin{equation*}\label{non-local-branching-mechanism}
\phi(x,f)=\beta(x)(f(x)-\eta(x,f)),
\end{equation*}
where $\beta\in B^+(E)$ and $\eta$ has representation
\begin{equation*}\label{non-local-zeta-representation}
\eta(x,f)=\gamma(x,f)+\int_{M(E)^{\circ}}(1-{\rm e}^{-\langle{f},{\nu}\rangle})\Gamma(x,\d \nu),
\end{equation*}
such that $\gamma(x,f)$ is a {\color{black}uniformly} bounded function on $E\times B^+(E)$ and $\langle{1},{\nu}\rangle\Gamma(x,\d \nu)$ is a  {\color{black} uniformly (for $x\in E$)} bounded kernel from $E$ to $M(E)^{\circ}:=M(E)\backslash \{0\}$ with
\begin{equation*}
\gamma(x,f)+\int_{M(E)^{\circ}}\langle{1},{\nu}\rangle\Gamma(x,\d \nu)\leq 1.
\end{equation*}
We refer the reader to \cite{dawson2002nonlocal, PY} for more details regarding the above formulae. 
Lemma 3.1 in \cite{dawson2002nonlocal} tells us that the functional $\eta(x,f)$ has the following equivalent representation
\begin{equation}\label{zeta-representation}
\eta(x,f)=\int_{M_0(E)}\bigg[\delta_\eta(x,\pi)\langle{f},{\pi}\rangle+\int_{0}^{\infty}(1-{\rm e}^{-u\langle{f},{\pi}\rangle})n_\eta(x,\pi,\d u)\bigg]P_\eta(x,\d\pi),
\end{equation}
where $M_0(E)$ denotes the set of probability measures on $E$, $P_\eta(x,\d\pi)$ is a probability kernel from $E$ to $M_0(E)$, $\delta_\eta\geq 0$ is a bounded function on $E\times M_0(E)$, and $un_\eta(x,\pi,\d u)$ is a bounded kernel from $E\times M_0(E)$ to $(0,\infty)$ with
\begin{equation*}
\delta_\eta(x,\pi)+\int_{0}^{\infty}un_\eta(x,\pi,\d u)\leq 1.
\end{equation*}

The reader will note that we have deliberately used some of the same notation for both branching Markov processes and  superprocesses. In the sequel there should be no confusion and the motivation for this choice of repeated notation is that our main results are indifferent to which of the two processes we are talking about.
\medskip

Let us now define what we mean by a $(\sP,\psi,\phi)$-superprocess with immigration. In order to do so, we need to introduce two objects, the first of which is the excursion measure for the $(\sP,\psi,\phi)$-superprocess. 
It is known, see \cite{DK04} or Chapter 8 of \cite{ZL11}, that a measure  $\mathbb{Q}_x$ exists on the space {\color{black}$\mathbb{D} =\mathbb{D}([0,\infty)\times \mathcal{M}(E))$} which satisfies
\[
\mathbb{Q}_x\left(1-{\rm e}^{-\langle f, X_t\rangle}\right)=\sV_t[f](x),
\]
for $x\in E,\;t\geq 0$ and $f\in B^+(E)$. The second object is the immigration mechanism, which we define, for $f\in B^+(E)$, via
\begin{equation}\label{chi}
\chi [f]= \langle f, \upsilon\rangle +\int_{M(E)^{\circ}}(1-{\rm e}^{-\langle{f},{\nu}\rangle})\Upsilon(\d \nu),
\end{equation}
where $\upsilon\in M(E)$ and $(1\wedge\langle 1, \nu\rangle)\Upsilon(\d \nu)$ is a finite measure on $M(E)^{\circ}$. 
As above, Lemma 3.1 of \cite{dawson2002nonlocal} enforces the necessity of the decomposition
\begin{equation*}\label{ups-representation}
\chi [f]=\int_{M_0(E)}\bigg[\delta_\chi(\pi)\langle{f},{\pi}\rangle+\int_{0}^{\infty}(1-{\rm e}^{-u\langle{f},{\pi}\rangle})n_\chi(\pi,\d u)\bigg]P_\chi(\d\pi),
\end{equation*}
where $\delta_\chi\geq 0$ is a bounded function on  $M_0(E)$, $un_\chi(\pi,\d u)$ is a bounded kernel from $M_0(E)$ to $(0,\infty)$ and $P_\chi(\d\pi)$ is a probability   on $M_0(E)$.

\begin{definition}[Non-local superprocess with immigration]
We say that  $Y^{(\mu)}=(Y^{(\mu)}_t, t\geq0)$ is a $(\sP,\psi,\phi)$-superprocess with $\chi$-immigration (or a $(\sP,\psi,\phi,\chi)$-SPI for short) with initial mass $\mu\in M(E)$, if 
\begin{equation}
Y^{(\mu)}_t = X^{(\mu)}_t  + \int_0^t \int_{\mathbb{D}} X_{t-s}{\sN}(\dd s, \dd X), \qquad t\geq0,
\label{SPI}
\end{equation}
where   ${\sN}(\dd s, \dd X)$ is Poisson random measure on $[0,\infty)\times \mathbb{D}$ with  intensity

\[
\int_{M_0(E)}P_\chi(\dd\pi)
\left(\delta_\chi(\pi)\int_{E}\pi(\dd y)\mathbb{Q}_{y}(\dd X) + {\int_0^\infty}n_\chi(  \pi, \dd u)\mathbb{P}_{u\pi}(\dd X)\right)\dd s.
\]
We also write $\mathbb{P} = (\mathbb{P}_{\mu}, \mu\in M(E))$ for the probabilities of $Y^{(\mu)}$.
%and ${\sN}^\Upsilon(\dd s, \dd x)$ is an optional random measure on $[0,\infty)\times \mathbb{D}$ with predictable compensator 
%\[
%\int_{M_0(E)}P_\chi(\xi_s,\dd\pi) n_\chi(\xi_s,  \pi, \dd u)\mathbb{P}_{u\pi}(\dd x)\dd s.
%\]
\end{definition}

We note that similar constructions for SPI processes can be found in  \cite{ZL2001, ZL2021}.

\section{Assumptions and main results}%\footnote{\azul Pedro: Section 3 has my approval with the changes I have introduced.}
Before stating our results, we first introduce some assumptions that will be crucial in analysing the models defined above. Unless a specific difference is indicated, the assumptions apply both to the setting that $X$ is either a non-local branching Markov process or a non-local superprocess.

\medskip

\noindent{\bf {(H1)}:}
We assume second moments
\begin{equation*}
  \sup_{x\in E} \mathcal{E}_x \left[N^2\right]<\infty \qquad \text{and} \qquad \sup_{x\in E} \left(\int_0^\infty y^2 \nu(x,\d y) + 
  \int_{M(E)^\circ} \langle 1, \nu\rangle^2  \Gamma(x,\d \nu)\right)<\infty.
\label{kmomsup}
\end{equation*}
Assumption {(H1)} allows us to define, for $f\in B^+(E)$
\begin{equation*}
    \mathbb{V}[f](x) = \beta(x)\mathcal{E}_x\left[\sum_{i, j= 1;\, i\neq j}^N f(x_i)f(x_j)\right], \qquad x\in E,
\end{equation*}
in the branching Markov process setting or with
\begin{align}\label{VVdef}
\mathbb{V}[f](x)  &= \psi''(x,0+)f(x)^2+ \beta(x)\int_{M(E)^\circ} \langle f, \nu\rangle^2 \Gamma(x,\d\nu),\\ 
&= \left(2c(x) +  \int_{0}^\infty y^2 \nu (x, \d y)\right)f(x)^2 + \beta(x)\int_{M(E)^\circ} \langle f, \nu\rangle^2 \Gamma(x,\d\nu),
\end{align}
in the superprocess setting. 

\medskip

\noindent{\bf {(H2)}:} There exist a constant $\lambda \leq 0$, a function $\varphi \in B^+(E)$ and finite measure {$\tilde\varphi\in M(E)$} such that, for $f\in B^+(E)$,
\[
\langle \sT_t[\varphi] , \mu\rangle = {\rm e}^{\lambda t}\langle{\varphi},{\mu}\rangle 
\text{ and } 
\langle{\sT_t[f] },{\tilde\varphi} \rangle= {\rm e}^{\lambda t}\proint{f}{\tilde\varphi},
\]
for all $\mu\in N(E)$ (resp. $M(E)$) if $(X, \mathbb{P})$ is a branching Markov process (resp. a superprocess),  where
    \[
        \langle\mathtt{T}_t[f], \mu\rangle = \int_{E}\mu(\dd x)\mathbb{E}_{\delta_x} \left[\langle f, X_t\rangle\right]= \mathbb{E}_{\mu} \left[\langle f, X_t\rangle\right], \qquad t\geq 0.
    \]
Further let us define
 \begin{equation}
 \Delta_t = \sup_{x\in E,\, f\in B^+_1(E)}|\varphi(x)^{-1}{\rm e}^{-\lambda t}\sT_t\bra{f}(x)-\proint{f}{\tilde{\varphi}}| , \qquad t\geq 0.
\label{Deltat}
 \end{equation}
 We suppose that $\Delta := \sup_{t\geq 0}\Delta_t <\infty$ and
\begin{equation*}
 \Delta_t=O({\rm e}^{-\varepsilon t})\text{ as }t\to\infty\text{ for some $\varepsilon>0$}. %\le C{\rm e}^{-\gamma t}, \qquad t \ge 0.  
\label{T1ass}
\end{equation*}
%{\color{black} Would the first of these not be automatic by the second and that $\sT_t[f]$ is right-continuous? the latter may need some additional conditions e.g. Feller.}
{Without loss of generality, we conveniently impose the normalisation $\langle\varphi,\tilde{\varphi}\rangle=1$.}
 
\begin{remark}\rm
    The non-local spatial branching process is known as critical (resp. subcritical) when $\lambda=0$ (resp. $\lambda<0$). Without restriction on the sign of $\lambda$, assumption {(H2)} has been recently {\color{black}named} (see \cite{Horton2023}) the Asmussen-Hering class of branching processes, {\color{black}acknowledging  foundational results for this class in \cite{AH1, AH2, AH3}.}
\end{remark}

 \noindent{\bf {(H3)}:}  For each $x\in E$ %all $\mu\in N(E)$ (resp. $\mu\in M((E)$) the process becomes extinct $\mathbb{P}_\mu$-almost surely, i.e. 
\[
\mathbb{P}_{\delta_x}(\zeta < \infty) = 1,
\]
where $\zeta = \inf \{t>0 : \langle 1, X_t\rangle = 0\}.$

\medskip

\noindent{\bf {(H4)}:}\label{H2} There exist constants $K> 0$ and $M>0$ such that for all $f \in B^+(E)$,  
    \begin{equation*}
      \langle \mathbb{V}_M[f], \tilde\varphi \rangle \ge K \langle  {f}, \tilde\varphi\rangle^2,
      \label{weirdmixing}
    \end{equation*}
where $\mathbb{V}_M$ is defined by
    \[
        \mathbb{V}_M[f](x) = \beta(x)\mathcal{E}_x\left[\sum_{i, j= 1;\, i\neq j}^N f(x_i)f(x_j)\mathbf{1}_{\{N\leq M\}}\right], \qquad x\in E,
    \]
for  branching Markov processes and by
    \begin{align*}
        \mathbb{V}_M[f](x)  &= \left(2c(x) +  \int_{0}^\infty y^2 \mathbf{1}_{\{y\leq M\}} \nu (x, \d y)\right)f(x)^2\notag\\
        &\hspace{2cm} + \beta(x)\int_{M(E)^\circ} \langle f, \nu\rangle^2 \mathbf{1}_{\{\langle 1, \nu\rangle\leq M\}}\Gamma(x,\d\nu),\qquad x\in E,
    \end{align*}
for superprocesses. 

Notice that in both cases, $\mathbb{V}[f] = \lim_{M\to\infty} \mathbb{V}_M[f]$ by monotone convergence.

\medskip

We are now ready to state our main results. The reader will note that the results are stated for both branching Markov processes and superprocesses simultaneously. Moreover, as alluded to in the introduction, the proofs of these results, whether with or without immigration, are similar in spirit and methodology. {\color{black} Accordingly, as the reader will see, we therefore opt to give all proofs in the BMP setting.}

The first two results pertain to the critical system, i.e. when $\lambda = 0$ in {(H2)}, and when there is no immigration. In particular, we show that the Kolmogorov survival probability asymptotic holds, as well as the Yaglom limit. In essence these results  support the notion of universality of the exponential distribution for the asymptotic law of $\langle f, X_t\rangle/t$ conditional on survival as $t\to\infty$. %{}%\footnote{\azul Pedro: It is not so much the fact that the critical results without immigration are a necessary step to obtain the critical results with immigration, but that the proof technique is similar in the critical results with and without  immigration.}

\begin{theorem}[Kolmogorov survival probability at criticality]\label{thm:survival}%\footnote{\azul Pedro: I rewrote the theorem using the initial mass $\mu$ instead of the previous formulation with $\sup_{x\in E}|\ldots|$.}
Suppose that $(X, \mathbb{P})$ is a $(\sP, \sG)$-BMP (resp. a $(\mathtt{P},\psi,\phi)$-SP) satisfying {(H1)}--{(H4)} with $\lambda=0$.  Then, for all $\mu\in N(E)$ (resp. $\mu\in M(E)$),
% We have
% \begin{equation*}
%  \lim_{t \to\infty} \sup_{x\in E} \left| \frac{t\mathbb{P}_{\delta_x} (\zeta>t)}{\varphi (x)} - \frac{2}{\proint{\mathbb{V}[\varphi]}{\tilde\varphi}}\right| = 0.
%  \label{eq:Kolmogorov}
%  \end{equation*}
\begin{equation*}
 \lim_{t \to\infty} t\mathbb{P}_{\mu} (\zeta>t) =  \frac{2\langle \varphi, \mu\rangle }{\langle \mathbb{V}[\varphi],\tilde\varphi\rangle}.
 \label{eq:Kolmogorov}
 \end{equation*}
\end{theorem}

\medskip

\begin{theorem}[Yaglom limit at criticality]\label{thm:Yaglom}
Suppose that $(X, \mathbb{P})$ is a $(\sP, \sG)$-BMP (resp. a $(\mathtt{P},\psi,\phi)$-SP) satisfying {(H1)}--{(H4)} with $\lambda=0$. 
% For $f \in B^+(E)$ and for all $x \in E$, 
% \begin{equation*}
% \lim_{t \to \infty}\mathbb E_{\delta_x}\left[ {\rm exp}\left(-\theta\frac{\langle f, X_t\rangle}{t} \right) \bigg| \zeta > t \right] = \frac{1}{1 + \tfrac12 \theta\langle  f, \tilde\varphi\rangle \langle  \mathbb V[\varphi], \tilde\varphi \rangle}.
% \label{eq:Yaglom}
% \end{equation*}
Then, for all $\mu\in N(E)$ (resp. $\mu\in M(E)$),
\begin{equation*}
\lim_{t \to \infty}\mathbb E_{\mu}\left[ {\rm exp}\left(-\theta\frac{\langle f, X_t\rangle}{t} \right) \Big|\, \langle 1, X_t\rangle >0 \right] = \frac{1}{1 +  \theta\tfrac12 \langle  f, \tilde\varphi\rangle \langle  \mathbb V[\varphi], \tilde\varphi \rangle},
\label{eq:Yaglom}
\end{equation*}
where $\theta\geq 0$ and $f\in B^+(E)$.
\end{theorem}

\medskip

%\begin{remark}\label{rem:survival-Yaglom}
%\rm Theorems \ref{thm:survival} and \ref{thm:Yaglom} were proved in the setting of non-local branching Markov processes in \cite{HHKW} albeit under the significantly more restrictive assumption that $\langle 1, \mathcal Z\rangle$ is bounded above almost surely. In the setting of superprocesses, there are related results in \cite{RSZ}, otherwise, as far as we are aware, these results are new in the setting of non-local superprocesses.
%\end{remark}

We now state the two results that provide conditions for the stability of the system {\it with} immigration at both criticality and subcriticality. The reader should note that Theorems \ref{thm-gamma-dist} and \ref{subcrit} do not assume hypotheses {(H3)} and {(H4)}. {\color{black}  In the spirit of the  setting without immigration, Theorem \ref{thm-gamma-dist}   supports the notion of universality of the gamma distribution for the asymptotic law of $\langle f, Y_t\rangle/t$ conditional on survival as $t\to\infty$.}

\begin{theorem}[Stability at criticality]\label{thm-gamma-dist}%\footnote{\azul Pedro: I have reviewed the theorem several times and hypotheses {\azul (H3)} and {\azul (H4)} are not used. This is a significant fact. Please check it yourself to ensure that I have not made any mistakes.}
Suppose that $(Y, \mathbb{P})$ is a $(\mathtt{P}, \mathtt{G},\mathtt{H})$-BMPI, resp. a $(\mathtt{P},\psi,\phi,\chi)$-SPI, satisfying  {{(H1)}} and {{(H2)}} with $\lambda=0$. 
% Then, for $f \in B^+(E)$ and for all $\mu \in N(E)$ (resp. $\mu \in M(E)$), 
% \begin{equation}
% \lim_{t \to \infty}\mathbb{E}_{\mu} \left[ \exp\left(-\theta \frac{\langle f, Y_t\rangle}{t} \right)\right] = \left(1 + \frac{1}{2} \theta\langle f, \tilde\varphi\rangle \langle  \mathbb V[\varphi], \tilde\varphi \rangle\right)^{-{2{\mathtt I}[\varphi]}/{\langle  \mathbb V[\varphi], \tilde\varphi \rangle}}
% \label{eq:Gamma}
% \end{equation}
 Then, for every $f\in B^+(E)$, the random variable $\langle f, Y_t\rangle /t$ converges weakly as $t\to\infty$ if and only if $\mathtt{I}[\varphi]<\infty$, where
 \[
 \mathtt{I}[\varphi] = \alpha \tilde{\mathcal{E}} [\langle \varphi, \Tilde{\mathcal{Z}}\rangle] ,\qquad\text{resp. } {\mathtt I} [\varphi] = \langle \varphi, \upsilon \rangle + \int_{M(E)^\circ} \langle \varphi, \nu \rangle \Upsilon (\mathrm{d}\nu),
 \]
 for the BMPI, resp.  SPI setting. In that case, for all $\mu\in N(E)$ (resp. $\mu\in M(E)$) and $\theta\geq 0$,
\begin{equation*}
\lim_{t \to \infty}\mathbb{E}_{\mu} \left[ \exp\left(-\theta \frac{\langle f, Y_t\rangle}{t} \right)\right] = \left(1 +  \theta\frac{1}{2}\langle f, \tilde\varphi\rangle \langle  \mathbb V[\varphi], \tilde\varphi \rangle\right)^{-{2{\mathtt I}[\varphi]}/{\langle  \mathbb V[\varphi], \tilde\varphi \rangle}}.
\label{eq:Gamma}
\end{equation*}
\end{theorem}

%
%\begin{theorem}[Stability at criticality]\label{crit} \footnote{\azul Should we remove this theorem?}Suppose that $(Y, \mathbb{P})$ is a  $(\sP, \sG, \mathtt{H})$-BMPI satisfying {\azul (H3)}--{\azul (H1)} with $\lambda=0$. Then there exists a measure $Y_\infty$ on $E$ such that $Y_t \to Y_\infty$ weakly as $t\to \infty$ if and only if
%%\[
%%    \int_0^{z_0} \frac{\mathtt{H}(z\varphi)}{z^2}\mathrm{d}z <\infty \text{ for some } z_0>0,
%%\]
%\begin{equation}\label{int-test-BPMI-crit}
%     \int_0^{z_0} \frac{\mathtt{H}[z\varphi]}{z^2}\mathrm{d}z <\infty \text{ for some } z_0>0,
%\end{equation}
%and in that case the stationary distribution is given by
%\begin{equation}\label{stat-dist-BMPI}
%    \mathbb{E}\left[\mathrm{e}^{-\langle f,Y_\infty\rangle}\right] = \exp \left(-\int_0^\infty \mathtt{H}[\mathtt{v}_{s}[f]] \mathrm{d}s\right).
%\end{equation}
%
%Now suppose that $(Y, \mathbb{P})$ is a $(\sP,\psi,\phi,\chi)$-SPI satisfying {\azul (H3)}--{\azul (H1)}  with $\lambda=0$. Then there exists a measure $Y_\infty$ on $E$ such that $Y_t \to Y_\infty$ weakly as $t\to \infty$ if and only if
%\begin{equation}\label{integral-test-SPI}
%    \int_0^{z_0} \frac{\chi[z\varphi]}{z^2}\mathrm{d}z <\infty \text{ for some } z_0>0,
%\end{equation}
%and in that case the stationary distribution is given by
%\begin{equation}\label{stat-dist-SPI}
%    \mathbb{E}\left[\mathrm{e}^{-\langle f,Y_\infty\rangle}\right] = \exp \left(-\int_0^\infty \chi(\mathtt{V}_{s}[f]) \mathrm{d}s\right).
%\end{equation}
%\end{theorem}

\begin{remark}
    Under {(H1)} and {(H2)}, we deduce from the above theorem that there is no stationary measure $Y_\infty$ on $E$ such that $Y_t \to Y_\infty$ weakly as $t\to \infty$. This is due to the fact that the process $(Y, \mathbb{P})$ always explodes at criticality.
\end{remark}

\begin{theorem}[Stability at subcriticality]\label{subcrit} Suppose that $(Y, \mathbb{P})$ is a  $(\sP, \sG, \mathtt{H})$-BMPI, resp. $(\sP,\psi,\phi,\chi)$-SPI,  satisfying   {(H1)} and {(H2)} with $\lambda<0$. Then, {for all $\mu\in N(E)$, resp. $\mu\in M(E)$}, there exists a measure $Y_\infty$ on $E$ given by 
%\begin{equation}\label{stat-dist-BMPI}
 \[
    \mathbb{E}_{\mu}\left[\mathrm{e}^{-\langle f,Y_\infty\rangle}\right] = {\rm e}^{-\int_0^\infty \mathtt{H}[\mathtt{v}_{s}[f]] \mathrm{d}s}, \text{ resp. } \mathbb{E}_{\mu}\left[\mathrm{e}^{-\langle f,Y_\infty\rangle}\right] = {\rm e}^{-\int_0^\infty \chi(\mathtt{V}_{s}[f]) \mathrm{d}s},
    \]
 such that $Y_t \to Y_\infty$ weakly as $t\to \infty$ if and only if
\begin{equation}
  \int_0^{z_0} \frac{\mathtt{H}[z\varphi]}{z}\mathrm{d}z <\infty, \qquad\text{resp. }   \int_0^{z_0} \frac{\chi[z\varphi]}{z}\mathrm{d}z <\infty,  \qquad\text{for some } z_0>0,
  \label{subinttest}
\end{equation}
if and only if
\begin{equation}\label{log-immig-BMPI}
     \Tilde{\mathcal{E}}\left[\log \left(1+\langle\varphi, \Tilde{\mathcal{Z}}\rangle\right)\right]<\infty, \qquad \text{resp. }     \int_{M(E)^{\circ}}\log \left(1+ \langle{\varphi},{\nu}\rangle\right)\Upsilon(\d \nu)<\infty.
\end{equation}
%Now suppose that $(Y, \mathbb{P})$ is a $(\sP,\psi,\phi,\chi)$-SPI satisfying  {\azul (H2)} and {\azul (H1)} with $\lambda<0$. Then there exists a measure $Y_\infty$ on $E$ given by %\eqref{stat-dist-SPI} 
%$\mathbb{E}\left[\mathrm{e}^{-\langle f,Y_\infty\rangle}\right] = \exp \left(-\int_0^\infty \chi(\mathtt{V}_{s}[f]) \mathrm{d}s\right)$
%such that $Y_t \to Y_\infty$ weakly as $t\to \infty$ if and only if
%\[
%  \int_0^{z_0} \frac{\chi[z\varphi]}{z}\mathrm{d}z <\infty, \text{ for some } z_0>0.
%\]
%if and only if
%\begin{equation}\label{log-immig-SPI}
%     \int_{M(E)^{\circ}}\log \left(1+ \langle{\varphi},{\nu}\rangle\right)\Upsilon(\d \nu)<\infty .
%\end{equation}
\end{theorem}

%\medskip
\begin{remark}\rm
    Notice that if $\inf_{x\in E} \varphi (x) > 0$, then the eigenfunction $\varphi$ can be substituted into conditions %\eqref{int-test-BPMI-crit}, \eqref{integral-test-SPI}, 
    \eqref{subinttest} and \eqref{log-immig-BMPI} by the constant function $1$.  This will be the case, for instance, in the continuous-time multi-type Galton--Watson model with immigration, where the eigenfunction is the Perron--Frobenius eigenvector of the offspring mean matrix. 
\end{remark}

%\footnote{\azul Pedro: I eliminated corollary 1 because it did not add anything new. Instead, I added remark 3 to point out that there is always explosion at criticality}

%\footnote{\azul Pedro: I removed corollary 2. The reason is that it was a heuristic type result that did not have a rigorous proof. Checking the proof was implicitly using that the Gateaux derivative ("directional" derivative) implies the Frechet derivative (differentiability in normed spaces). This is not true in general.}

\section{Discussion}
In this section, we spend some time discussing the consistency of our results with the existing literature. Moreover, we also take the opportunity to discuss assumptions (H1)--(H4) in the setting of some specific spatial processes.

\subsection{Consistency with known results}%\footnote{\azul Pedro: In the same way as with the introduction, subsection 4.1 has to undergo a small modification to adjust to the new content of the article.}
Theorems \ref{thm:survival} and \ref{thm:Yaglom} 
are the analogues of the Kolmogorov and Yalgom limit theorems which are so classical that they barely need any introduction. Needless to say, one may find them included in the standard branching process texts Athreya and Ney \cite{AN} and Asmussen and Hering \cite{AH3}.
Both Theorems \ref{thm:survival} and \ref{thm:Yaglom}  were recently proved in the setting of non-local branching Markov processes in \cite{HHKW} albeit under the significantly more restrictive assumption that $\langle 1, \mathcal Z\rangle$ is bounded above almost surely {\color{black} by a constant}. In the setting of superprocesses, {\color{black} the limits in Theorems \ref{thm:survival} and \ref{thm:Yaglom}} are studied \cite{RSZ} for local branching mechanisms but, as far as we are aware, these results are new in the setting of non-local superprocesses.

With regards to Theorem \ref{thm-gamma-dist}, as alluded to above, the analogue of \eqref{subinttest} in the form of the classical integral test \eqref{classicalintegral} is well know for {subcritical} Galton--Watson processes as well as corresponding to log-moments of immigrating mass at each generation. 
{At criticality, the convergence towards a gamma distribution after normalization has also been widely studied; we refer the reader to \cite{AN, FW,  Heathcote1965, Heathcote1966, Pakes-1971, Seneta1968, Seneta1970} for discrete time results and  \cite{Pakes1975, Seva1957, Yang} for continuous time.
For models with continuous mass, the natural analogues of Galton--Watson processes are continuous-state branching process. For this setting, the picture was first described by \cite{Pinsky} with further detail given in \cite{Mijatovic}. See also Chapter 3 of \cite{ZL11}, where an integral similar to that of \eqref{subinttest} can also be found  for the setting of CSBPs.
%\footnote{\azul Pedro: Is the gamma result proven for CBIs? I wasn't able to find it.}

Finally, regarding Theorem \ref{subcrit}, \cite{Quine, Kaplan} provide results which mirror those of Theorem \ref{subcrit} for subcritical multi-type Galton--Watson processes with immigration. It is worth noting that multi-type branching processes may be considered as one of the simplest examples of a spatial branching process, where the spatial component evolves in a discrete or finite set. In this setting, the mean matrix of types across a single generation codes the notion of criticality through the value of its leading eigenvalue in relation to unity. (The assumption {(H2)} is a direct generalisation of this concept.) We were unable to find any continuous-time analogues in the setting of multi-type Galton--Watson processes, hence we presume that Theorem \ref{subcrit} is a new result in this setting given that they are special cases of BMPIs as we have defined them.

Otherwise, for general BMPIs, we are unaware of any work on immigration which covers the level of generality addressed in Theorems \ref{thm-gamma-dist} and  \ref{subcrit}. For the setting of  SPIs, the most comprehensive work to date that we could find is nicely summarised in  Section 9.6 of \cite{ZL11}; see also  \cite{ZL2001, ZL2021} for related material. Nonetheless, we note that e.g. the integral test and the scaled limit to a gamma distribution we provide appear to be new. 

On a final note, we mention that the setting for general BMPIs and SPIs has some implicit context through the well understood study of martingale change of measures in a variety of settings, see e.g. \cite{Evans, Etheridge, EK} among many others. As a rule of thumb, it is known that inherent additive martingales, which  typically arise from the leading right-eigenfunction described in assumption {(H2)}, when used as a change of measure on the ambient probability space,  invoke a so-called spine decomposition, which is akin to a BMPI/SPI. Although distributional stability of the spatial population is not necessarily of concern in this context (whereas martingale convergence typically is), the notion of controlled growth through logarithmic moment conditions is certainly an important part of the dialogue. For a general perspective of martingale changes of measure and immigration in the context of BMPIs, see the discussion in \cite{Horton2023}.

\subsection{Two examples}%\footnote{\azul Pedro: For me the examples are still valid and correct. Changes are not necessary.}
We  consider two concrete examples which resonate with existing literature.
\medskip

\noindent{\bf Branching Brownian motion  on a compact domain.}
We consider a regular branching Brownian motion in which particles branch independently at a constant rate $\beta>0$ with i.i.d. numbers of offspring distributed like $N$. Branching is local in the sense that offspring are positioned at the point in space where their parent dies. This process is contained in a regular  bounded domain $D$ such that when an individual first touches the boundary of the domain it is killed and sent to a cemetery state. This model was considered by \cite{Ellen}, for which it was shown that {(H2)} holds providing $\partial D$ is Lipschitz. It is also known that for subcritial and critical systems, as defined by {(H2)}, the requirement {(H3)} holds. In fact, it is necessarily the case that $\tilde\varphi = \varphi$. Moreover, as soon as $\mathcal{E}[N^2]<\infty$, we also have that {(H1)} and {(H4)} hold. Indeed, for the latter, it is easy to see that 
\[
        \mathbb{V}_M[g](x) = \beta g(x)^2 \mathcal{E}\left[N(N-1)\mathbf{1}_{\{N\leq M\}}\right], \qquad x\in E,
\]
so that 
\[
\langle \mathbb{V}_M[g], \tilde\varphi\rangle = \beta  \mathcal{E}\left[N(N-1)\mathbf{1}_{\{N\leq M\}}\right]\langle g^2, \tilde\varphi \rangle \geq K \langle g, \tilde\varphi\rangle^2
\]
by Jensen's inequality, which implies \eqref{weirdmixing} holds. 

\medskip

\noindent{\bf Multi-type continuous-state branching processes (MCSBP).} These processes are the natural analogues of multi-type Galton--Watson processes in the context of continous mass. One may also think of them as super Markov chains. In addition, allowing for immigration, MCSBPs were introduced in \cite{BLP} and can be represented via their semigroup properties or as solutions to SDEs. In essence, they correspond to the setting that $E = \{1,\ldots, n\}$, for some $n\in\mathbb{N}$. In this setting, {(H2)} is a natural ergodic assumption similar to those discussed in Section 4 of \cite{KP}, where a simple irreducibility assumption ensures that {(H2)} will hold. In essence, {(H2)} is the  classical Perron--Frobenius behaviour for the matrix of the mean semigroup.  The assumption {(H3)} does not automatically hold as, in the spirit of  CSBP processes, extinction can occur by a slow trickle of mass down to zero. The assumption {(H1)} is also natural,  ensuring finite second moments for the MCSBP mass. Finally, for the assumption {(H4)}, {
\begin{align*}
        \mathbb{V}_M[h](i)  &= \left(2c(i) +  \int_{(0,\infty)} y^2 \mathbf{1}_{\{y\leq M\}} \nu (i, \d y)\right)h(i)^2\notag\\
        &\hspace{2cm} + \beta(i)\int_{M(\{1,\ldots,n\})^\circ} \langle h, \nu\rangle^2 \mathbf{1}_{\{\langle 1, \nu\rangle\leq M\}}\Gamma({i},\d\nu),\qquad i\in \{1,\ldots,n\},
    \end{align*}
so again by Jensen's inequality it is easy to see that
\begin{equation*}
    \langle \mathbb{V}_M[h], \Tilde{\varphi}\rangle \geq \min_{i\in \{1,\ldots,n\} } \left(2c(i) +  \int_{(0,\infty)} y^2 \mathbf{1}_{\{y\leq M\}} \nu (i, \d y)\right) \langle h^2, \Tilde{\varphi}\rangle \geq K \langle h, \Tilde{\varphi}\rangle^2, 
\end{equation*}
i.e. {(H4)} is automatically satisfied.
}

\section{Evolution equations}%\footnote{\azul Pedro: For me in sections 5, 5.1, 5.2 everything is correct and does not need changes.}
In this section, we consider several semigroup evolution equations that will be useful for proving our main results. We note that, in formulating them, we don't need to assume as much as {(H1)}--{(H4)}. {\color{black}We recall that the assumptions on the branching mechanisms $\mathtt{G}$, $\psi$ and $\phi$ ensure  that
	\[
	  \sup_{x \in E}\mathcal E_x[N] < \infty,  \qquad \text{resp. }    \sup_{x \in E}\left( \int_0^\infty |y| \nu(x, {\rm d}y) + \int_{M(E)^\circ} \langle 1, \nu\rangle \Gamma(x, {\rm d}
	  \nu)  \right) < \infty,
	\]
for branching particle processes, resp. superprocesses, which is needed for some of the results we  cite below. }

%{\color{blue}This assumption is needed for the proofs of the re-oriented evolution equation containing the operators $\mathtt A$ / $\mathtt J$. Seems strange to need this but not finite first moments of the immigration mechanism. Do we also need this somewhere?}

\subsection{Non-local branching Markov processes}
In the setting of the $(\sP, \sG)$-branching Markov process, the evolution equation for the expectation semigroup $(\sT_t, t \ge 0)$ is given by
\begin{equation*}\label{eq:BMP-linear}
\sT_t[f](x) = \sP_t[f](x) + \int_0^t \sP_s[\beta (\sm[\sT_{t-s}[f]] - \sT_{t-s}[f])](x) \d s,
\end{equation*}
for $t \ge 0$, $x \in E$ and $f \in B^+(E)$, where we have used the notation
\begin{equation*}\label{eq:mean} 
\sm[f](x) =  \mathcal E_x[\langle f, \mathcal Z\rangle].
\end{equation*}
See, for example, Lemma 8.1 of \cite{Horton2023}. 

\medskip

Our next evolution equation will relate the non-linear semigroup  to the linear semigroup, which will enable us to use {(H2)}, for example, to study the limiting behaviour of $\sv_t$. For this, we will introduce the following modification to the non-linear semigroup,
\[
  \su_t[g](x) = \mathbb E_{\delta_x}\left[1-\prod_{i = 1}^{N_t}g(x_i(t)) \right] = 1-\sv_t[-\ln g](x), \qquad g\in B^+_1(E).
\]
Recalling that we have assumed first moments of the offspring distribution, one can show that
\begin{equation}\label{eq:nonlin-lin}
\su_t[g](x) = \sT[1-g](x) - \int_0^t \sT_s[\sA[\su_{t-s}[g]]](x) \d s, \quad t \ge 0,
\end{equation}
where, for $g \in B^+_1(E)$ and $x \in E$, 
% \begin{equation}
% \sA[h](x) = \beta(x)\mathcal E_x\left[1 - \prod_{i = 1}^N (1-h(x_i)) - \sum_{i = 1}^Nh(x_i)\right].
% \label{A}
% \end{equation}
\begin{equation}%\footnote{\azul Pedro: I changed the sign in the definition of $\sA$ so that it is non-negative, just as it happens with the operator $\sJ$}
\sA[g](x) = \beta(x)\mathcal E_x\left[\prod_{i = 1}^N (1-g(x_i)) - 1 + \sum_{i = 1}^N g(x_i)\right].
\label{A}
\end{equation}
We refer the reader to Theorem 8.2 of \cite{Horton2023} for a more general version of \eqref{eq:nonlin-lin}, along with a proof.

\medskip

We now consider the process {\it with} immigration. Let us consider the transition semigroup  $(\mathtt{w}_t, t\geq0)$ for the $(\sP,\mathtt{G},\mathtt{H})$-BMPI, given by
\begin{equation}\label{eq: semigroup BPI}
    {\rm e}^{-\mathtt{w}_t[f](x)} = \mathbb{E}_{\delta_x}\left[{\rm e}^{-\langle f, Y_t\rangle}\right], 
\qquad t\geq0, \, f\in B^+ (E), \, x\in E.
\end{equation}
Denoting $\Tilde{Y}_t = \sum_{j=1}^{D_t} X_{t-\tau_j}^{(\Tilde{\mathcal{Z}}_j)}$, from Definition \ref{def-BMPI} it is clear that $Y^{(\delta_x)}_t = X^{(\delta_x)}_t  + \Tilde{Y}_t$ and 
$$
{\rm e}^{-\mathtt{w}_t[f](x)} = {\rm e}^{-\mathtt{v}_t[f](x)}{\rm e}^{-\Tilde{\mathtt{w}}_t[f]},
$$ 
where $(\Tilde{\mathtt{w}}_t,t\geq 0)$ is the transition semigroup associated to $\Tilde{Y}$. From the branching property \eqref{MBP} and the immigration counting measure \eqref{immigr-counting-measure}, it is clear that the Laplace functional of $X_t^{(\Tilde{\mathcal{Z}})}$ is given by
\begin{equation*}
    \Tilde{\mathcal{E}} \left[{\rm e}^{-\langle \mathtt{v}_t[f], \tilde{Z}\rangle}\right], \qquad f\in B^+(E).
\end{equation*}
Therefore, conditioning on the time of the first immigration event, it is possible to obtain
\begin{align*}
    \mathrm{e}^{-\Tilde{\mathtt{w}}_t [f]} &= \mathrm{e}^{-\alpha t} + \int_0^t  \alpha\mathrm{e}^{-\alpha s} \Tilde{\mathcal{E}} \left[\mathrm{e}^{-\langle\mathtt{v}_{t-s}[f], \tilde{Z}\rangle}\right] \mathrm{e}^{-\Tilde{\mathtt{w}}_{t-s} [f]}   \mathrm{d}s \\
    &= \mathrm{e}^{-\alpha t} + \int_0^t  \mathrm{e}^{-\alpha s} \left[\alpha-\mathtt{H}[\mathtt{v}_{t-s}[f]]\right] \mathrm{e}^{-\Tilde{\mathtt{w}}_{t-s} [f]}   \mathrm{d}s .
\end{align*}
It then follows from Theorem 2.1 in \cite{Horton2023} or Lemma 1.2 in Chapter 4 of \cite{D02}, for example, that
$$
\mathrm{e}^{-\Tilde{\mathtt{w}}_t [f]} = 1 - \int_0^t   \mathtt{H}[\mathtt{v}_{t-s}[f]]\mathrm{e}^{-\Tilde{\mathtt{w}}_{t-s} [f]}    \mathrm{d}s ,
$$
from which it is easily deduced that
\begin{equation}\label{evol-eq-BMPI}
    \mathrm{e}^{-\mathtt{w}_{t}[f](x)} = \exp\left(-\mathtt{v}_{t}[f](x) - \int_0^t \mathtt{H}[\mathtt{v}_{s}[f]] \mathrm{d}s\right).
\end{equation}
{\color{black}We note that similar calculations for the semigroup of processes with immigration are common  in other literature, e.g. in Chapters 3 and 9 of    \cite{ZL11}.}

\subsection{Non-local superprocesses}
In the setting of the  $(\sP,\psi,\phi)$-superprocess,
the evolution equation for the expectation semigroup $(\sT_t, t\geq0)$ is well known and satisfies 
\begin{equation}
\sT_t\bra{f}(x)=\sP_t[f](x)%-\int_{0}^{t}\sP_s\bra{(\beta-b)\sT_{t-s}[f]}(x)\d s
+\int_{0}^{t}\sP_s\bra{\beta (\sm[\sT_{t-s}[f]]-\sT_{t-s}[f])+b\sT_{t-s}[f]} (x)\d s,
\label{superm21}
\end{equation} 
for $t\geq 0$, $x\in E$ and $f\in B^+(E)$, where, with a meaningful abuse of our branching Markov process notation, we now define
\begin{align}\label{eq-m}
\sm[f](x)%&=\int_{M_0(E)}\bra{\gamma(x,\pi)\langle{f},{\pi}\rangle+\int_{0}^{\infty}u\langle{f},{\pi}\rangle n(x,\pi,\d u)}G(x,\d\pi)\notag\\
&=\gamma(x,f) + \int_{M(E)^\circ} \langle f, \nu\rangle\Gamma(x,\d\nu).
\end{align}
See for example  equation (3.24) of \cite{dawson2002nonlocal}.

\medskip

Similarly to the branching Markov process setting, let us re-write {\color{black} an extended version of the non-linear semigroup evolution $({\sV}_t, t\geq0)$, defined in \eqref{non-local-evolution-equation}, i.e. the natural analogue of \eqref{nonlinv},} in terms of the linear semigroup $(\sT_t, t\geq0)$. 
{\color{black}To this end, define
\[
{\rm e}^{-{\sV}_t\bra{f}(x)}=\mathbb{E}_{\delta_x}\big[{\rm e}^{-\langle{f},{X_t}\rangle}\big].
\]
From \cite{GHK}, we have the following evolution equation, 
\begin{equation}
{\sV}_t[f](x)= \sT_{t}[f](x) - \int_{0}^{t}\sT_{s}\left[\sJ[{\sV}_{t-s}[f]]\right](x)\d s, \quad f \in B^+(E), x \in E, t\geq 0,
\label{nonlinvJ}
\end{equation}
where, for $h\in B^+(E)$ and $x\in E$, we now define
\begin{align}
\sJ[h](x) &= \psi(x, h(x)) + \phi(x, h) + \beta(x)(\sm[h](x)-h(x)) + b(x)h(x)\notag\\
&=c(x)h(x)^2 +\int_{(0,\infty)}( {\rm e}^{-h(x) y}-1+h(x) y)\nu(x,\d y)\notag\\
&\hspace{1cm} +\beta(x)\int_{M(E)^{\circ}}({\rm e}^{-\langle{h},{\nu}\rangle}-1+ \langle h, \nu\rangle)\Gamma(x,\d \nu).
\label{J}
\end{align}
}

\medskip

As before, we now consider the non-local superprocess with immigration. In particular, we are interested in the transition semigroup pair $((\sW_t,\sV_t), t\geq0)$ for the $(\sP,\psi,\phi, \chi)$-SPI, where 
\begin{equation}\label{eq-semigr-SPI}
    {\rm e}^{-\sW_t[f](x)} = \mathbb{E}_{\delta_x}\left[{\rm e}^{-\langle f, Y_t\rangle}\right], 
\qquad t\geq0.
\end{equation}
From the definition \eqref{SPI}, with the help of  Campbell's formula, we have
\begin{align}
{\rm e}^{-\sW_t[f](x)} &= {\rm e}^{-\sV_t[f](x)} \exp\left(-\int_0^t \int_{M_0(E)}P_\chi(\dd\pi)\delta_\chi( \pi)\int_{E}\pi(\dd y)
\mathbb{Q}_y(1-{\rm e}^{-\langle f, X_{ t-s}\rangle}){\dd s}\right)\notag\\
&\hspace{2.5cm}\times \exp\left(-\int_0^t \int_{M_0(E)}P_\chi( \dd\pi){\int_0^\infty}n_\chi(  \pi, \dd u)
\mathbb{E}_{u\pi}(1-{\rm e}^{-\langle f, X_{ t-s}\rangle}){\dd s}\right) \notag\\
&={\rm e}^{-\sV_t[f](x)}  \exp\left(- \int_0^t \left[ {\langle \sV_{t-s}[f] ,\upsilon\rangle}
+ \int_{M^\circ(E)} (1- {\rm e}^{-\langle \sV_{t-s}{[f]}, \nu\rangle})\Upsilon(  \dd\nu)\right]\dd s\right) \notag\\
&=  \exp\left(-\sV_t[f](x)- \int_0^t \chi[\sV_{t-s}[f]]
\dd s\right).
\label{evol-eq-SPI}
\end{align}
% \begin{align*}
% {\rm e}^{-\sW_t[f](x)} &= {\rm e}^{-\sV_t[f](x)} \exp\left(-\int_0^t \int_{M_0(E)}P_\chi(\dd\pi)\delta_\chi( \pi)\int_{E}\pi(\dd y)
% \int_{\mathbb{D}}(1-{\rm e}^{-\langle f, X_{\azul t-s}\rangle})\mathbb{Q}_y(\dd X){\azul\dd s}\right)\notag\\
% &\hspace{2.5cm}\times \exp\left(-\int_0^t \int_{M_0(E)}P_\chi( \dd\pi){\azul\int_0^\infty}n_\chi(  \pi, \dd u)
% \int_{\mathbb{D}}(1-{\rm e}^{-\langle f, X_{\azul t-s}\rangle})\mathbb{P}_{u\pi}(\dd X){\azul\dd s}\right) \notag\\
%  &= {\rm e}^{-\sV_t[f](x)} \exp\left(-\int_0^t \int_{M_0(E)}P_\chi(\dd\pi)\delta_\chi( \pi)\int_{E}\pi(\dd y)
% \mathbb{Q}_y(1-{\rm e}^{-\langle f, X_{\azul t-s}\rangle}){\azul\dd s}\right)\notag\\
% &\hspace{2.5cm}\times \exp\left(-\int_0^t \int_{M_0(E)}P_\chi( \dd\pi){\azul\int_0^\infty}n_\chi(  \pi, \dd u)
% \mathbb{E}_{u\pi}(1-{\rm e}^{-\langle f, X_{\azul t-s}\rangle}){\azul\dd s}\right) \notag\\
% &={\rm e}^{-\sV_t[f](x)}  \exp\left(- \int_0^t  \int_{M_0(E)}P_\chi(\dd\pi)\left[\delta_\chi( \pi)\langle \sV_{t-s}[f], \pi\rangle
% + {\azul\int_0^\infty}n_\chi(  \pi, \dd u)
% (1-{\rm e}^{-u\langle \sV_{t-s}[f], \pi\rangle})\right]\dd s\right) \notag\\
% &={\rm e}^{-\sV_t[f](x)}  \exp\left(- \int_0^t \left[ {\azul\langle \sV_{t-s}[f] ,\upsilon\rangle}
% + \int_{M^\circ(E)} (1- {\rm e}^{-\langle \sV_{t-s}{\azul [f]}, \nu\rangle})\Upsilon(  \dd\nu)\right]\dd s\right) \notag\\
% &=  \exp\left(-\sV_t[f](x)- \int_0^t \chi[\sV_{t-s}[f]]
% \dd s\right).
% \label{evol-eq-SPI}
% \end{align*}
We note that similar calculations can be found in Chapter 9 of \cite{ZL11}.
\section{Proof of Theorem \ref{thm:survival}}

%After the discussion in Remark \ref{rem:survival-Yaglom}, we prefer to give a detailed proof for the framework of superprocesses. We follow a strategy very close to that of \cite{HHKW}. In the setting of particle systems, the steps are very similar so we will omit some parts.

\subsection{Non-local branching Markov processes}\label{KolBMP}
Let us define $\su_t(x) := \su_t[{\bf 0}](x)$, where ${\bf 0}$ is the constant $0$ function. Then $\su_t(x) = \mathbb P_{\delta_x}(\zeta > t)$ and hence $\mathbb P_{\delta_x}(\zeta > t)$ is a solution to \eqref{eq:nonlin-lin} with $f = {\bf 0}$. Our aim will be to use this evolution equation to obtain the asymptotic behaviour for $\su_t(x)$. In order to do so, it will be convenient to introduce the following quantities,
\[
  a_t[g] := \langle \su_t[g], \tilde\varphi \rangle, \quad \text{and} \quad a_t := a_t[{\bf 0}] = \langle  \su_t ,\tilde\varphi\rangle.
\] 
Integrating \eqref{eq:nonlin-lin} with respect to $\tilde\varphi$ and using the {(H2)}, we obtain
\begin{equation}\label{eq:at}
  a_t[g] = \langle 1-g  ,\tilde\varphi\rangle - \int_0^t \langle  \sA[\su_{t-s}[g]], \tilde\varphi\rangle\d s.
\end{equation}
The strategy of the proof is to first find coarse upper and lower bounds of the order $1/t$ for $a_t$ and $\su_t$, and then refine our estimates to obtain the precise constants. The method of proof follows closely that of \cite{HHKW} but with more precise estimates in our calculations. For the convenience of the reader, we will include the details. 

\medskip

We thus proceed via a series of lemmas, the first of which provides useful lower bounds on $\su_t[g]$ and $a_t[g]$ for general $g$. For the following lemma, we introduce the following change of measure
\begin{equation}
\frac{\d \mathbb P_\mu^\varphi}{\d\mathbb P_\mu}\bigg|_{\mathcal F_t} := \frac{\langle \varphi, X_t\rangle}{\langle \varphi, \mu\rangle}, \quad t \ge 0, \mu \in N(E),
\label{E:COM}
\end{equation}
where it follows from {(H2)} that $\langle \varphi, X_t\rangle / \langle \varphi, \mu\rangle$ is a martingale. 

\begin{lemma}\label{lem: a-lowerbd}
There exists $C \in (0, \infty)$ such that 
\[
\su_t[g](x)\ge \frac{C\varphi(x)}{\mathbb{E}^{\varphi}_{\delta_x}[\langle \varphi, X_t \rangle] +\sup_{y\in E}\frac{\varphi(y)}{\log(1/g(y))}}  \quad \text{ and }  \quad 
a_t[g]\ge \frac{C}{\mathbb{E}^{\varphi}_{\delta_x}[\langle \varphi, X_t \rangle] +\sup_{y\in E}\frac{\varphi(y)}{\log(1/g(y))}}
\]
for all $t\ge 0$ and $g\in B_1^+(E)$ such that $\sup_{y\in E}{\varphi(y)}/{\log(1/g(y))}<\infty$.

In particular, 
\[
\su_t(x)\ge \frac{C\varphi(x)}{t} \quad \text{ and } \quad a_t \ge \frac{C}{t}, \qquad t\ge 1.
\]
\end{lemma}

\begin{proof}
Recall the change of measure \eqref{E:COM}. 
 By Jensen's inequality, we have
\begin{equation}
\mathbb{E}_{\delta_x}\left[1-\prod_{i=1}^{N_t}g(x_i(t))\right]
=\varphi(x)\mathbb{E}^{\varphi}_{\delta_{x}}\left[\frac{ 1-\prod_{i=1}^{N_t}g(x_i(t))}{\langle \varphi, X_{t}\rangle}\right]\ge \frac{\varphi(x)}{\mathbb E^{\varphi}_{\delta_{x}}\big[\frac{\langle\varphi, X_{t}\rangle}{1-\prod_{i=1}^{N_t}g(x_i(t))}\big]}
\end{equation}
where we note that $1-{\rm e}^{-x}\ge \min(x/2,1/2)$ for $x\ge 0$, so that
\[
1-\prod_{i=1}^{N_t}g(x_i(t)) \ge \tfrac12 \min(\langle \log(1/g),X_t\rangle, 1),
\]
(where this also holds for $g=\mathbf{0}$ with the convention that $\log(1/0)=\infty$).
Thus 
\[
\mathbb{E}_{\delta_x}\left[1-\prod_{i=1}^{N_t}g(x_i(t))\right]
\ge \frac{2\varphi(x)}{\mathbb{E}_{\delta_x}^{\varphi}\big[\max \{\langle \varphi, X_t\rangle,\frac{\langle \varphi, X_t \rangle}{\langle \log(1/g),X_t \rangle}\big]}
\ge \frac{2\varphi(x)}{\mathbb{E}^{\varphi}_{\delta_x}[\langle \varphi, X_t \rangle] + \sup_{y\in E}\frac{\varphi(y)}{\log(1/g(y))}}.
\]

The lower bound for $a_t[g]$ then follows from an integration with $\tilde\varphi$, recalling that we have normalised the left and right eigenfunctions so that $\langle \varphi , \tilde{\varphi}\rangle =1$.

The specific claim when $g=\mathbf{0}$ follows from \cite[Theorem 1]{GHK}, since this implies that $\mathbb{E}^{\varphi}_{\delta_x}[\langle \varphi, X_t \rangle] \le Ct$ for some constant $C\in (0,\infty)$ and for all $t\ge 1$. 
\end{proof}

\medskip

The next lemma shows that the leading order term in $\sA$, defined in \eqref{A}, is governed by the operator  $\mathbb V$.

\begin{lemma}[Properties of $\sA$ and $\mathbb{V}$]\label{L:Gprops}{Under the assumption {(H1)} we have the following.}
\begin{enumerate}[label=(\alph*), ref=\ref{L:Gprops}\alph*]\setlength{\itemsep}{0em}
%\item[]
\item We have $\sA[g](x)\geq 0 \text{ for all } g\in B^+_1,\, x\in E.$ 
\label{Gneg}
\item There exists a constant $C\in (0,\infty)$ such that $
\|\mathbb{V}[h_1]-\mathbb{V}[h_2]\| \le C\|h_1-h_2\|$ for all functions $h_1, h_2\in \{f\in B^+(E): \|f\|\leq k\}$ for some $k>0$.
\label{Vcont}
\item $\sA[g](x)\le \tfrac{1}{2} \mathbb{V}[g](x)$ for all $g\in B_1^+(E)$, $x\in E$.
\label{Gmod}
\item  $\|\sA[g]-\tfrac12 \mathbb{V}[g]\| =  o(\|g\|^2)$ as $\|g\|\to 0$, $g\in B_1^+(E)$. 
\label{GVsmall}
\end{enumerate}
\end{lemma}

\begin{proof}
(a) is a consequence of the deterministic inequality \[0\le \prod_{i=1}^N (1-z_i)-1+\sum_{i=1}^N z_i  \] for all $N\in \mathbb{N}$ and $z_1,\dots, z_N\in [0,1]$.

For (b), we write
\begin{align*}
    \|\mathbb{V}[h_1]-\mathbb{V}[h_2]\| &\leq \|\beta \| \sup_{x\in E}  \mathcal{E}_x\left[\sum_{i,j=1;\, i\ne j}^N \big|[h_1(x_i)-h_2(x_i)]h_1(x_j)+h_2(x_i)[h_1(x_j)-h_2(x_j)]\big|\right]
\end{align*}
and the claim follows thanks to {(H1)}.

Finally, we deduce (c) and (d). Using Taylor's theorem, we have
		\begin{equation}\label{E:Taylor}
			\sA[h](x)=\beta(x)\mathcal{E}_x\bigg[\sum_{\substack{i,j=1, \\ i \ne j}}^N h(x_i)h(x_j) \int_0^1 (1-r) \prod_{\substack{k=1, \\ k\ne i,j}}^N (1-rh(x_k)) \mathrm{d}r \bigg].
		\end{equation}
	This immediately implies (c), and we also deduce that 
	\begin{align*}
			\|\sA[h]-\tfrac12\mathbb{V}[h]\| 
			& \le \sup_{x\in E} \beta(x) \mathcal{E}_x\left[\sum\nolimits_{i\ne j} h(x_i)h(x_j) \int_0^1(1-r)\bigg|1-\prod\nolimits_{k\ne i,j}(1- r h(x_k)\bigg| \, \mathrm{d}r \right]\nonumber \\
			& \le \|\beta\| \|h\|^2 \sup_x \mathcal{E}_x[N(N-1)(N\|h\| \mathbf{1}_{\{N\le \|h\|^{-1/2}\}}+2\mathbf{1}_{\{N\ge \|h\|^{-1/2}\}})].
		\end{align*}
	This is $o(\|h\|^2)$ thanks to {(H1)}.
\end{proof}

With this in hand, we proceed to show that $\su_t[g]/\varphi$ and $a_t[g]$ are small when $\su_t[g]$ is small, which will be crucial for the proof of Theorem \ref{thm:survival}. 

\begin{lemma} There exists $C\in (0,\infty)$ such that for all $t>0$ and all $g\in B_1^+(E)$,
	\label{lem: iter}
	\[
	\sup_{x \in E} \left|\frac{\su_t[g](x)}{\varphi(x)} - a_t[g] \right|  \le 
	C \left({\rm e}^{-\varepsilon t}\|1-g\| + \int_0^t {\rm e}^{-\varepsilon(t-s)}\|\su_s[g]\|^2 \d s \right).
	\]
\end{lemma}

\begin{proof} We have
	\begin{align*}
		\left|\frac{\su_t[g](x)}{\varphi(x)}-a_t[g]\right| & \le \left|\frac{\sT_t[1-g](x)}{\varphi(x)}-\langle 1-g,\tilde\varphi\rangle \right| +\int_0^t \left| \frac{\sT_{t-s}[\sA[\su_s[g]]](x)}{\varphi(x)} - \langle  \sA[\su_s[g]] ,\tilde\varphi\rangle \right| \, \mathrm{d} s \\ 
		& \le C\left( {\rm e}^{-\varepsilon t}\|1-g\| + \int_0^t {\rm e}^{-\varepsilon (t-s)}\|\sA[\su_s[g]]\| \, \mathrm{d}s\right),
	\end{align*}
	where the second line follows from {(H2)}. The lemma then follows since $\|\sA[h]\| \le \tfrac{1}{2}  \|\mathbb{V}[h]\|$ (Lemma \ref{Gmod}) and $\|\mathbb{V}[\su_s[g]]\| \le \sup\nolimits_{x\in E}\mathcal{E}_x[N(N-1)]\|\beta\| \|\su_s[g]\|^2$, which is  thanks to Assumption {(H1)}.
\end{proof}

\begin{lemma}\label{lem: unif-small} 
Under the assumptions of Theorem \ref{thm:survival}, there exists $t_0\in (0,\infty)$ and a constant $C >0 $ such that for all $t\ge t_{0}$ and all $g\in B^+_1(E)$,
\begin{equation}
   a_t \le \frac{C}{t}\,\text{ and } \,\sup_{x \in E}\su_{t}(x) \le \frac{C}{t}.
   \label{eq: a-upperbd BMPI}
\end{equation}
\end{lemma}

\begin{proof}
We first observe from {(H3)} that 
\[
  a_t \to 0 \, \text{ and }  \,\sup_{x \in E}\su_{t}(x) \to 0
\]
as $t \to \infty$. The convergence of $a_t$ to $0$ follows from {(H3)} %\footnote{\azul Pedro: This is the only place where {\azul (H3)} is used for BMPI.}, 
and dominated convergence. To prove uniform convergence of $\su_{t}$ to zero, let us note that $\su_{t+s}(x)=\su_{t}[1-\su_{s}](x)$ by the Markov branching property. We therefore have (also using Lemma \ref{Gneg}) that
\begin{equation}\label{eq: ut-bd}
0\le \su_{t+s}(x)=\sT_{t}[\su_{s}](x)-\int_{0}^{t}\sT_{l}\big[\sA[\su_{t+s-\ell}]\big](x)\d \ell\le \sT_{t}[\su_{s}](x)
\end{equation}
and so 
\begin{equation}\label{E:ut-a-bd}
	 \|\su_{t+s}\| \le \| \sT_t[\su_s] \| \le a_s\|\varphi\|+O({\rm e}^{-\varepsilon t}) 
	 \end{equation}
by {(H2)}. Taking $t$ and then $s$ to infinity gives that $\|\su_t\| \to 0$ as $t\to \infty$.

Now we prove the required upper bound on $a_t$ and $\|\su_t\|$. First note that \eqref{eq:at} implies that 
\[
a_t=a_0+\int_0^t \langle  \sA[u_s], \tilde\varphi \rangle \mathrm{d}s
\]
where the integrand is bounded due to Lemma \ref{Gneg}. Therefore $a$ is differentiable with $a'_t=-\langle \sA[\su_t], \tilde\varphi \rangle$ for $t\ge 0$. 

Next, by Taylor's Theorem \eqref{E:Taylor}, we deduce that if $\|h\|\le 1/2$, then $ \sA[h](x)\geq 2^{-M} \mathbb{V}_M[h](x)$ for any $M\in \mathbb{N}$.
We therefore obtain that for $t\ge t_0$, with $t_0$ chosen so that $\sup\nolimits_{t\ge t_0}\|\su_{t}\| \le 1/2$, 
\[
a'(t)=-\langle  \sA[\su_t], \tilde\varphi \rangle \le -2^{-M}\langle  \mathbb{V}_M[\su_t] ,\tilde\varphi\rangle \le -C \langle \su_t, \tilde{\varphi}\rangle^2 = -C a_t^2 
\]
for some $C\in (0,\infty)$, where in the second inequality we have used {(H4)}
%\footnote{\azul Pedro: This is the only place where {\azul (H4)} is used for BMPI.} 
with the values of $M$ and $C$ there.

This yields
\[
\frac{\mathrm{d}}{\mathrm{d}t}\left(\frac{1}{a_t}\right) \ge C \text{ for } t\ge t_0.
\]
Integrating from $t_0$ to $t$ we obtain the desired upper bound for $a_t$. The upper bound for $\su_t$ follows from the same argument as that given in \eqref{E:ut-a-bd}.
\end{proof}

%The next result shows that the long-term behaviour of $\su_t/\varphi$ and $a_t$ are the same, which will be key to obtaining the correct constants in the bounds obtained in the previous lemma.

We are now ready to prove Theorem \ref{thm:survival}, which now entails showing that the long-term behaviour of $\su_t/\varphi$ and $a_t$ are the same.

\begin{proof}[Proof of Theorem \ref{thm:survival}]
Applying Lemma \ref{lem: iter}, we have 
\begin{equation*}
    \sup_{x \in E}\left|\frac{\su_{t}(x)}{\varphi(x)}-a_t\right|  \le 
	C \left({\rm e}^{-\varepsilon t} + \int_0^{t/2} {\rm e}^{-\varepsilon(t-s)}\|\su_s\|^2 \d s + \int_{t/2}^t {\rm e}^{-\varepsilon(t-s)}\|\su_s\|^2 \d s \right).
\end{equation*}
Bounding $\|\su_s\|$ by 1, we see that $\int_0^{t/2} {\rm e}^{-\varepsilon(t-s)}\|\su_s\|^2 \d s = O({\rm e}^{-\varepsilon t/2})$, and using the bound given in Lemma \ref{lem: unif-small}, we obtain $\int_{t/2}^t {\rm e}^{-\varepsilon(t-s)}\|\su_s\|^2 \d s = O(t^{-2})$. Therefore,
\[
\sup_{x \in E}\left|\frac{\su_{t}(x)}{\varphi(x)}-a_t\right|=O(t^{-2}), \quad t\to\infty. 
\]
On the other hand, from Lemma \ref{lem: a-lowerbd}, we have that $a_t^{-1}=O(t)$. It follows that 
\begin{equation}
\label{eq: u_t-a_t}
\sup_{x \in E}\left|\frac{\su_{t}(x)}{\varphi(x)a_t}-1\right|=O(t^{-1}), \quad t\to\infty. 
\end{equation}
Using  Lemma \ref{Vcont},
we deduce that
\begin{align}\notag
\sup_{x \in E}a_t^{-2}\left|{\V}[\su_{t}](x)-{\V}[a_t\varphi](x)]\right|&=\sup_{x \in E}\left|{\V}[{\su_{t}}/{a_t}](x)-{\V}[\varphi](x)\right| \\ \label{eq: bd-V}
&\le C \sup_{x \in E}\left|\frac{\su_{t}(x)}{a_t}-\varphi(x)\right| =O(t^{-1}). 
\end{align}
Therefore, appealing to basic calculus from \eqref{eq:at}, for all $t\ge t_{0}$,
\[
\left|\frac{1}{ta_t}-\frac{1}{ta_{t_{0}}}-\langle \tfrac12 \mathbb{V}[\varphi], \tilde\varphi \rangle \right|  = \frac 1t \left|\int_{t_{0}}^{t} \frac{\langle  \sA[\su_{s}],\tilde\varphi\rangle}{a_s^2} \d s - \int_0^t \langle \tfrac12 \mathbb{V}[\varphi], \tilde\varphi \rangle \d s \right|.\]
Noting that $\mathbb{V}[\varphi a_s] = a_s^2\mathbb{V}[\varphi]$,  we can bound the right-hand side above by 
\[\frac1t\left| \int_0^{t_0} \langle \tfrac12 \mathbb{V}[\varphi], \tilde\varphi \rangle \right| 
+ \frac1t \left| \int_{t_0}^t \frac{\tfrac12 \langle \mathbb{V}[a_s\varphi]- \mathbb{V}[\su_s],\tilde\varphi\rangle}{a_s^2}  \d s\right| 
+ \frac1t \left| \int_{t_0}^t \frac{\langle \tfrac12\mathbb{V}[\su_s]- \sA[\su_s],\tilde\varphi \rangle }{a_s^2}  \d s\right|.\]

The first term in the expression above clearly converges to $0$ as $t\to \infty$, while the second term converges to $0$ using  \eqref{eq: bd-V}. The final term converges to $0$ again using the lower bound $a_s\ge C/s$ from Lemma \ref{lem: a-lowerbd} and Lemma \ref{GVsmall}.

Since $1/(ta_{t_0})\to 0$ as $t\to \infty$, this implies that 
\[ \frac{1}{ta_t} \to \langle \tfrac12 \mathbb{V}[\varphi],\tilde\varphi \rangle, \quad \text{ as } t\to \infty, \] 
in other words,
\[
a_t\sim \frac{2}{ \langle  \mathbb{V}[\varphi] ,\tilde\varphi\rangle t},\quad \text{ as }t\to\infty.
\]
 The desired asymptotic for $\su_{t}$ then follows from \eqref{eq: u_t-a_t}. 
\end{proof}

\subsection{Non-local Superprocesses}

%\footnote{\azul Pedro: I have rewritten the test to follow exactly the same steps as in section 6.1 and be consistent, in the sense that we can treat both processes in the same way.}

We first note that, for $\theta\in\mathbb R$ and $\mu\in M(E)$,
\begin{equation}
{\rm e}^{-\langle\sV_t, \mu\rangle}: = \lim_{\theta\to\infty} {\rm e}^{-\langle\sV_t[\theta],\mu\rangle}= \lim_{\theta\to\infty}\mathbb{E}_\mu\big[{\rm e}^{-\theta\langle{1},{X_t}\rangle}\big] = \mathbb{P}_{\mu}(\zeta\leq t),
\label{thetalimit}
\end{equation}
and hence
\begin{equation}
 \mathbb{P}_{\mu}(\zeta> t) = 1-{\rm e}^{-\langle\sV_t, \mu\rangle}, \qquad \mu\in M( E),\, t\geq0.
 \label{Vtprob}
\end{equation}
Under assumption {(H3)}, %\footnote{\azul Pedro: First place where it is used {\azul (H3)} for SPI}, 
we notice that $\langle\sV_t, \mu\rangle\to0$ as $t\to\infty$ and consequently \eqref{Vtprob} implies that 
\begin{equation}
\lim_{t\to\infty}\frac{1}{\langle\sV_t, \mu\rangle} \mathbb{P}_{\mu}(\zeta> t) = 1.
\label{vtnotut}
\end{equation}
Thus, in order to understand the decay of the survival probability, it suffices to study the decay of $\sV_t$. 
For this reason, we will conveniently work with 
\begin{equation}
a_{t}[f] = \langle \sV_t [f], \tilde\varphi\rangle = \langle f, \tilde\varphi\rangle  - \int_{0}^{t} \langle \sJ[{\sV}_{t-s}[f]] , \tilde\varphi\rangle\d s, \quad f \in B^+(E),
\label{atint}
\end{equation}
where the second equality follows from integrating \eqref{nonlinvJ} with respect to $\tilde\varphi$ and {(H2)}. It follows that for any $t,t_0>0$,
\begin{equation}
a_{t+t_0}[f] = a_{t_0}[f]  - \int_{0}^{t} \langle \sJ[{\sV}_{s}[\sV_{t_0}[f]]] , \tilde\varphi\rangle\d s.
\label{atint2}
\end{equation}
Thus, with $a_t : = \lim_{\theta \to \infty}a_t[\theta] = \langle \sV_t, \tilde{\varphi}\rangle$, it follows that 
\begin{equation}
a_{t+t_0} = a_{t_0}   - \int_{0}^{t} \langle \sJ[{\sV}_{s}[\sV_{t_0}]] , \tilde\varphi\rangle\d s.
\label{atint3}
\end{equation}

The strategy is thus to prove that $a_t$ and $\sV_t$ are asymptotically of order $1/t$, as in the proof methodology of \cite{HHKW}. 

The proof {\color{black} of Theorem \ref{thm:survival}}  for the superprocess setting is almost verbatim the same as in the previous section. In the interests of brevity, we leave the remainder of the proof of Theorem \ref{thm:survival} for non-local superprocesses as an exercise, offering as assistance below the analogue of Lemma \ref{L:Gprops}, which is a key ingredient, and referring for a full proof {\color{black} of Theorem \ref{thm:survival}}  to \cite{thesis}.

\begin{lemma}[Properties of $\sJ$ and $\mathbb{V}$]\label{L:Jprops} Suppose that assumption {(H1)} holds. 
\begin{enumerate}[label=(\alph*), ref=\ref{L:Jprops}\alph*]\setlength{\itemsep}{0em}
%\item[]
\item We have $\mathtt{J}[h](x) \geq 0$ for all $h\in B^+(E)$, $x\in E.$ 
\label{Jpos}
\item There exists a constant $C\in (0,\infty)$ such that $
\|\mathbb{V}[h_1]-\mathbb{V}[h_2]\| \le C\|h_1-h_2\|$ for all functions $h_1, h_2\in \{f\in B^+(E): \|f\|\leq k\}$ for some $k>0$.
\label{VcontSPI}
\item $\mathtt{J}[h](x)\le \tfrac{1}{2} \mathbb{V}[h](x)$ for all $h\in B^+(E)$, $x\in E.$ 
\label{Jmod}
\item  $\|\mathtt{J}[h]-\tfrac12 \mathbb{V}[h]\| =  o(\|h\|^2)$ as $\|h\|\to 0$, $h\in B^+(E)$.  
\label{JVsmall}
\end{enumerate}
\end{lemma}

\begin{proof}
(a) and (c) follows trivially from the deterministic inequalities \[0 \leq \mathrm{e}^{-z}-1+z \leq \frac{z^2}{2}, \qquad z\in[0,\infty)]. \] 

For (b), we write 
\begin{align*}
    \|\mathbb{V}[h_1]-\mathbb{V}[h_2]\| &\leq  \sup_{x\in E} \left(2c(x) + \int_0^\infty y^2 \nu (x,\dd y) \right)\|h_1+h_2\|\|h_1-h_2\| \\
    &\quad + \sup_{x\in E} \beta(x)\int_{M(E)^\circ} \langle \|h_1+h_2\|, \nu\rangle \langle \|h_1-h_2\|, \nu\rangle | \Gamma(x,\d\nu)
\end{align*}
and the claim follows from {(H1)}.

For (d), the map ${m_{\star}}: z\in [0,\infty) \to 1-z+\frac{1}{2}z^2-\mathrm{e}^{-z}$ is a non-negative increasing function bounded above by $z^2$ and $z^3$, which allows us to write
\begin{align*}
    |\mathtt{J}[h](x)- \tfrac{1}{2} \mathbb{V}[h](x)| &\leq \int_0^\infty {m_{\star}}(\|h\| y) \nu (x,\dd y) + \beta(x)\int_{M(E)^\circ} {m_{\star}}(\langle \|h\|, \nu\rangle ) \Gamma(x,\d\nu) \\
    &\leq \|h\|^2 \int_0^\infty y^2 (\|h\|^{1/2} \mathbf{1}_{\{y\leq \|h\|^{-1/2}\}} + \mathbf{1}_{\{y\geq \|h\|^{-1/2}\}} ) \nu (x,\dd y) \\
    &\quad + \|h\|^2 \beta(x)\int_0^\infty \langle 1,\nu\rangle^2 (\|h\|^{1/2} \mathbf{1}_{\{\langle 1,\nu\rangle\leq \|h\|^{-1/2}\}} + \mathbf{1}_{\{\langle 1,\nu\rangle\geq \|h\|^{-1/2}\}} ) \Gamma (x,\dd \nu)
\end{align*}
	 and this is $o(\|h\|^2)$ thanks to {(H1)}.
\end{proof}

\section{Proof of Theorem \ref{thm:Yaglom}}
%\subsection{Non-local branching Markov processes}\footnote{\azul Pedro: I added some details to Emma's proof}
We give the proof only for the setting of non-local branching Markov processes, again noting that {\color{black}the proof in the setting} of non-local superprocesses is almost verbatim, taking account of the fact that the role of $\su_t$ is played by $\sV_t$. The reader is again referred to \cite{thesis} for a full proof.

%\begin{proof}[Proof of Theorem \ref{thm:Yaglom}]

We are guided by  the argument in the proof of Theorem 1.3 in \cite{HHKW}. Given $f\in B^+(E)$, let us consider $\Tilde{f} = f - \langle f, \tilde\varphi \rangle \varphi$. The first assertion is that $\langle \Tilde{f}, X_t\rangle/t$ converges weakly under $\mathbb{P}_\mu (\cdot \mid \zeta >t)$ to 0 as $t\to\infty$. Indeed, by Markov's inequality,
\begin{equation*}
    \mathbb{P}_\mu \left(\left.\frac{|\langle \Tilde{f}, X_t\rangle |}{t} >\epsilon \right| \zeta >t\right) \leq \frac{\langle \mathtt{T}_t^{(2)}[\Tilde{f}],\mu\rangle}{\epsilon^2 t^2 \mathbb{P}_\mu \left(\langle 1, X_t\rangle >0\right)},
\end{equation*}
where $ \mathtt{T}_t^{(2)}[\tilde{f}](x) = \mathbb{E}_{\delta_x}[\langle \tilde{f}, X_t\rangle^2]$, for $t\geq0$, $x\in E$.
Moreover, the asymptotic behaviour of $t \mathbb{P}_\mu \left(\langle 1, X_t\rangle >0\right)$ as $t\to\infty$ is given in Theorem \ref{thm:survival}. Now, from Theorem 1 of \cite{GHK},
\begin{equation}
\label{Delta2}
    \frac{1}{t} \langle \mathtt{T}_t^{(2)}[\Tilde{f}],\mu\rangle \leq (\langle\Tilde{f},\Tilde{\varphi}\rangle^2 \langle\mathbb{V}[\varphi],\Tilde{\varphi}\rangle + \|\Tilde{f}\|^2 \Delta_t^{(2)})\langle \varphi , \mu \rangle,
\end{equation}
where 
$\Delta_t^{(2)} = \sup_{x\in E, \, f\in B^+_1(E)} |s^{-1}\varphi(x)^{-1}\sT_t^{(2)}\bra{f}(x)-\proint{f}{\tilde{\varphi}}^2{\proint{\mathbb{V}[\varphi]}{\tilde\varphi}}|$, which tends to zero as $t\to \infty$ from the aforementioned theorem.
%the definition of $ \Delta_t^{(2)}$ here is not important other to note that it was proved in the aforementioned theorem that $\Delta_t^{(2)}$ vanishes as $t\to \infty$.
Hence, as  $\langle\Tilde{f},\Tilde{\varphi}\rangle = 0$, we see that the right-hand side of \eqref{Delta2} tends to zero as $t\to\infty$.

Then, applying Slutsky's Theorem, it is enough to show that
\begin{equation*}
\lim_{t \to \infty}\mathbb E_{\mu}\left[ {\rm exp}\left(-\frac{\theta}{t} \langle f, \Tilde{\varphi}\rangle\langle \varphi, X_t\rangle \right) \bigg| \zeta > t \right] = \frac{1}{1 + \tfrac12 \theta\langle f,\tilde\varphi\rangle \langle \mathbb V[\varphi],\tilde\varphi \rangle},
\end{equation*}
or equivalently,
\begin{equation}\label{eq:slutsky BMPI}
\lim_{t \to \infty}\frac{\mathbb E_{\mu}\left[1- {\rm exp}\left(-{\theta} \langle f, \Tilde{\varphi}\rangle\langle \varphi, X_t\rangle/{t} \right) \right]}{\mathbb{P}_\mu \left(\langle 1, X_t\rangle >0\right)} = \frac{\tfrac12 \theta\langle f,\tilde\varphi\rangle \langle \mathbb V[\varphi],\tilde\varphi \rangle}{1 + \tfrac12 \theta\langle f,\tilde\varphi\rangle \langle \mathbb V[\varphi],\tilde\varphi \rangle}.
\end{equation}

Fix $\Tilde{\theta}\in (0,\infty)$ and define
\[ g_t(x):={\rm e}^{-\frac{\Tilde{\theta} \varphi(x)}{t}}. \]
We note that for $0\le s\le t$
\begin{align*}
\su_s[g_t](x)&=\mathbb{E}_{\delta_x}[1-\exp(-{\Tilde{\theta} \langle {\varphi , X_s} \rangle}/{t})]\notag\\
&
=\varphi(x)\mathbb{E}_{\delta_x}^\varphi\left[\frac{1-\exp(-{\Tilde{\theta} \langle {\varphi , X_s} \rangle}/{t})}{\langle {\varphi , X_s} \rangle }\right]\notag\\
&
=\frac{\Tilde{\theta}\varphi(x)}{t}\mathbb{E}_{\delta_x}^\varphi\left[\frac{1-\exp(-{\Tilde{\theta} \langle {\varphi , X_s} \rangle}/{t})}{{\Tilde{\theta}}\langle {\varphi , X_s} \rangle/{t} }\right].
\end{align*}
Since 
$ x^{-1}(1-{\rm e}^{-x})=\int_0^1 {\rm e}^{-ux} \, \d u$
for $x>0$, this yields that 
\begin{equation}\label{E:usgupper}
\su_s[g_t](x) = \frac{\Tilde{\theta}\varphi(x)}{t}\int_0^1 \mathbb{E}_{\delta_x}^\varphi[
\exp(- u {\Tilde{\theta} \langle {\varphi , X_s} \rangle}/{t})] \d u \le \frac{\Tilde{\theta} \varphi(x)}{t}, \qquad \forall s\in [0,t].
\end{equation}

We have 
\begin{equation}
	\left|\frac{1}{ta_t[g_t]}-\frac{1}{ta_0[g_t]}-\langle \tfrac12  \mathbb{V}[\varphi],\tilde\varphi \rangle \right|  = \frac 1t \left|\int_{0}^{t} \frac{\langle {{ \sA }[\su_{s}[g_t]],  \tilde\varphi}\rangle}{a_s[g_t]^2} \d s - \int_0^t \langle \tfrac12  \mathbb{V}[\varphi], \tilde\varphi \rangle \d s \right|  
	\label{modinverse}
	\end{equation}
	which we can bound above by 
	\begin{equation}
	\frac1t \left| \int_0^t \frac{\tfrac12\langle \mathbb{V}[a_s[g_t]\varphi]- \mathbb{V}[\su_s[g_t]],\tilde\varphi\rangle}{a_s[g_t]^2}  \d s\right| 
	+ \frac1t \left| \int_0^t \frac{\langle \tfrac12 \mathbb{V}[\su_s[g_t]]- \sA[\su_s[g_t]],\tilde\varphi \rangle }{a_s[g_t]^2}  \d s\right|.
	\label{E:agtbounds BMPI}
	\end{equation}
The second term above can be identified as  $t^{-1} \int_0^t o(\|\su_s[g_t]\|^2) (a_s[g_t])^{-2}\dd s$ by Lemma \ref{GVsmall}, and so converges to $0$ as $t\to \infty$, using \eqref{E:usgupper} and Lemma \ref{lem: a-lowerbd}. For the first term in \eqref{E:agtbounds BMPI}, we note that 
	\[
 \left| \frac{\tfrac12\langle \mathbb{V}[a_s[g_t]\varphi]- \mathbb{V}[\su_s[g_t]],\tilde\varphi\rangle}{a_s[g_t]^2}  \right| 
= \left| \tfrac12\langle \mathbb{V}[\varphi]- \mathbb{V}[{\su_s[g_t]}/{a_s[g_t]}],\tilde\varphi\rangle \right|
\]
which, by Lemma \ref{Vcont}, is bounded above by a finite constant times
\begin{equation}
\Big\|\frac{\su_s[g_t]}{a_s[g_t]}-\varphi\Big\| \leq \frac{C}{a_s[g_t]}\left({\rm e}^{-\varepsilon s}\|1-g_t\| + \int_0^s {\rm e}^{-\varepsilon(s-r)} \|\su_r[g_t]\|\dd r\right).
\label{extrastar}
\end{equation}
Finally using the lower bound Lemma \ref{lem: a-lowerbd} on $a_s[g_t]$, the fact that $\|1-g_t\| \le C/t$ and the upper bound \eqref{E:usgupper}, we obtain that 
\[
\left| \frac{\tfrac12\langle \mathbb{V}[a_s[g_t]\varphi]- \mathbb{V}[\su_s[g_t]],\tilde\varphi\rangle}{a_s[g_t]^2}  \right| 
\le C\left({\rm e}^{-\varepsilon s}+ \frac{1}{t}\right)
\]
for some $C\in (0,\infty)$ and all $s\in (0,t)$. This means that the second term in \eqref{E:agtbounds BMPI} also converges to $0$ as $t\to \infty$. 

From \eqref{eq:at} we have 
\[ \frac{1}{ta_0[g_t]} \to \frac{1}{\Tilde{\theta}} \text { as } t\to \infty,\]
and, hence, we obtain from \eqref{modinverse} that 
\begin{equation} \frac{1}{ta_t[g_t]}\to \frac{\langle \mathbb{V}[\varphi],\Tilde{\varphi}\rangle}{2}+\frac{1}{\Tilde{\theta}}
\label{usedjustbelow}
\end{equation}
as $t\to \infty$.
 
But as we {\color{black}see  from \eqref{extrastar}, \eqref{E:usgupper} and \eqref{usedjustbelow}}, 
\begin{equation*}
    \Big\|\frac{\su_t[g_t]}{a_t[g_t]}-\varphi\Big\| \leq C\left({\rm e}^{-\varepsilon t}+ \frac{1}{t}\right)
\end{equation*}
and therefore 
\begin{equation*}
    \lim_{t\to\infty} \frac{t \langle\mathtt{u}_t[g_t],\mu\rangle}{\langle\varphi,\mu\rangle} = \frac{\Tilde{\theta}}{1+\frac{1}{2}\Tilde{\theta}\langle \mathbb{V}[\varphi],\Tilde{\varphi}\rangle}.
\end{equation*}
This is precisely the proof of \eqref{eq:slutsky BMPI} {\color{black} when we take} $\Tilde{\theta}=\theta\langle f,\Tilde{\varphi}\rangle$, where we have used also the statement of Theorem \ref{thm:survival}.
\hfill $\square$

\section{Proof of Theorem \ref{thm-gamma-dist}}

%\subsection{Non-local Branching Markov Processes with Immigration}
%\footnote{\azul Pedro: I would like you to carefully check the test to ensure that I have not made any errors.}
Once again, we give only the proof in the setting of non-local branching Markov processes, and leave the setting of non-local superprocesses with the assurance that the proof is almost verbatim and that the full proofs can be found in \cite{thesis}.

The proof in the non-local branching Markov process setting is quite long, and so we break it into several steps.

%\footnote{\azul Pedro: first part, necessary and sufficient condition for the existence of limit distribution}
\medskip

\noindent{\bf Step 1:} We start by the considering  necessary and sufficient condition for the existence of a limiting distribution. 
Recalling \eqref{evol-eq-BMPI}, assuming that $\lim_{t\to\infty} \langle \sv_t[{\theta}{f}/{t} ],\mu\rangle = 0$, it is enough to prove that the limit
\begin{equation}\label{limit-gamma}
       \lim_{t\to \infty} \int_0^t {\mathtt{H}}[\sv_{s}[{\theta}f/{t} ]] \dd s
   \end{equation}
converges if and only if 
\begin{equation*}
{\mathtt I} [\varphi] = \alpha \tilde{\mathcal{E}} [\langle \varphi, \Tilde{\mathcal{Z}}\rangle]<\infty.
\end{equation*}

We start by looking for a functional upper bound  for $\sv_{s}[{\theta f}/{t} ]$, for any $s\leq t$. To this end, recall $\Delta_s$ and $\Delta: = \sup_{s\geq0}\Delta_s$ from {(H2)}. We have
    \begin{equation*}
        \sv_{s}[{\theta} f/ {t}](x) \leq \sT_{s}[{\theta}f/{t} ](x) \leq \left(\frac{\theta}{t}\langle  f,\tilde{\varphi}\rangle  + \frac{\theta}{t}\| f\|\Delta_s\right) \varphi(x) \leq \left(\langle 1,\tilde{\varphi}\rangle  + \Delta\right) \frac{\theta}{t}\| f\|\varphi(x),
    \end{equation*}
    where we used Jensen's inequality for the first inequality and {(H2)} for the second. This tells us in particular that 
    \begin{equation}
    \lim_{t\to\infty} \sup_{s\leq t}\langle \sv_s[{\theta}{f}/{t} ],\mu\rangle = 0.
    \label{vt/t}
    \end{equation}

To verify that ${\mathtt I} [\varphi] <\infty$ is a sufficient condition for \eqref{limit-gamma} to hold, notice that $\mathtt{H} [\cdot]$ and $\mathtt{v}_t [\cdot]$ are monotone in the sense that if $f,g\in B^+(E)$ with $f\leq g$, then $\mathtt{H} [f] \leq \mathtt{H} [g]$ and $\mathtt{v}_t [f] \leq \mathtt{v}_t [g]$ for all $t\geq 0$. Thus, by Monotone Convergence Theorem,
\begin{align*}
    \lim_{t\to\infty} \int_0^t {\mathtt{H}}[\sv_{s}[{\theta f}/{t} ]] \dd s &\leq \lim_{t\to\infty} t{\mathtt{H}}\left[\left(\langle 1,\tilde{\varphi}\rangle  + \Delta\right) {\theta}\| f\|\varphi/{t}\right]\notag\\
     &= \lim_{t\to\infty} t \alpha \Tilde{\mathcal{E}} \left[1-\exp\Big(-\frac{1}{t}\left(\langle 1,\tilde{\varphi}\rangle  + \Delta\right) \theta\| f\|\langle \varphi, \Tilde{\mathcal{Z}}\rangle\Big)\right] \\
    &=\left(\langle 1,\tilde{\varphi}\rangle  + \Delta\right) \theta\| f\|\mathtt{I}[\varphi],
\end{align*}
so $\mathtt{I} [\varphi]<\infty$ is a sufficient condition for the existence of the limit distribution.

We also claim that $\mathtt{I} [\varphi]<\infty$ is a necessary condition for the convergence of \eqref{limit-gamma}. Indeed, for all $s > 0$, $x\in E$ and $g\in B^+(E)$, using $\mathrm{e}^{-y}\leq 1-y+\frac{1}{2}y^2$ if $y\geq 0$, we have 
\begin{equation}\label{ineq_non-lin_moment-2}
   \mathrm{e}^{-\mathtt{v}_s [g] (x)} \leq 1 - \mathtt{T}_s [g] (x) + \frac{1}{2} \mathtt{T}_s^{(2)} [g] (x),
\end{equation}
where $\mathtt{T}_s^{(2)} [g] (x) = \mathbb{E}_{\delta_x} [\langle g , X_s \rangle^2]$ is well defined due to assumption {(H1)}; see \cite{GHK}. The asymptotic behaviour at criticality of the first two moments is known due to {(H2)} and Theorem 1 in \cite{GHK}. From those results, we can  state the following inequalities for the moments,
\begin{align}
   \mathtt{T}_s [g] (x) &\geq \left(\langle g,\tilde{\varphi}\rangle  - \|g\|\Delta_s\right) \varphi(x), \label{ineq-moment-1}\\
   \mathtt{T}_s^{(2)} [g] (x) &\leq s\left(\langle g,\tilde{\varphi}\rangle^2 \proint{\mathbb{V}[\varphi]}{\tilde\varphi} + \|g\|^2\Delta_s^{(2)}\right) \varphi(x).\label{ineq-moment-2}
\end{align}
where $\Delta_s^{(2)} = \sup_{x\in E, \, f\in B^+_1(E)} \left|s^{-1}\varphi(x)^{-1}\sT_s^{(2)}\bra{f}(x)-\proint{f}{\tilde{\varphi}}^2{\proint{\mathbb{V}[\varphi]}{\tilde\varphi}}\right|\to 0$ as $s\to \infty .$ Applying the three inequalities to $g=t^{-1}h$ with $t\geq 1$ and $h\in B^+(E)$,  we get
\begin{equation}
   \mathrm{e}^{-\mathtt{v}_s [t^{-1}h] (x)} \leq 1 - \frac{1}{t}\left(\langle h,\tilde{\varphi}\rangle  - \|h\|\Delta_s\right) \varphi(x) +\frac{s}{2t^2}\left(\langle h,\tilde{\varphi}\rangle^2 \proint{\mathbb{V}[\varphi]}{\tilde\varphi} + \|h\|^2\Delta_s^{(2)}\right)\varphi (x).
\end{equation}
For the terms involving $\Delta_s$ and $\Delta_s^{(2)}$, recall that $\lim_{s\to\infty} (\|h\|\Delta_s + \frac{1}{2} \|h\|^2 \Delta_s^{(2)})=0$, so for all $\varepsilon\in (0,\frac{1}{4})$ there exists $t_0\geq 1$ such that $\|h\|\Delta_s + \frac{1}{2} \|h\|^2 \Delta_s^{(2)} \leq  \varepsilon \langle h,\tilde{\varphi}\rangle$ if $s\geq t_0$. For the other terms, choose a constant $k\in (0,1)$ small enough such that $k\langle 1,\tilde{\varphi}\rangle\proint{\mathbb{V}[\varphi]}{\tilde\varphi}\leq 1.$ Then it is clear that $\langle h,\tilde{\varphi}\rangle^2 \proint{\mathbb{V}[\varphi]}{\tilde\varphi} \leq \langle h,\tilde{\varphi}\rangle$ if $\|h\|\leq k$. Combining this, we obtain
\begin{align*}
    \exp(-\mathtt{v}_s [t^{-1}h] (x)) &\leq
    1+\frac{\left(\frac{1}{2}\langle h,\tilde{\varphi}\rangle^2 \proint{\mathbb{V}[\varphi]}{\tilde\varphi}-\langle h,\tilde{\varphi}\rangle\right)+\left(\|h\|\Delta_s + \frac{1}{2} \|h\|^2 \Delta_s^{(2)}\right)}{t}\varphi(x) \\
    &\leq 1+\frac{-\frac{1}{2}\langle h,\tilde{\varphi}\rangle+\varepsilon \langle h,\tilde{\varphi}\rangle}{t}\varphi(x) \leq 1-\frac{\langle h,\tilde{\varphi}\rangle\varphi(x)}{4t},
\end{align*}
for $t\geq s\geq t_0$ and $h\in B^+(E)$ with $\|h\|\leq k$. Let us use this final expression. For all $t\geq s\geq  t_0$ and $f\in B^+(E)$, we have $t^{-1}(k\wedge \theta f)\leq t^{-1} \theta f $ and
\begin{equation*}
   \exp(-\mathtt{v}_s [t^{-1} \theta f] (x)) \leq    \exp(-\mathtt{v}_s [t^{-1}(k\wedge \theta f)] (x)) \leq 1-\frac{\langle k\wedge \theta f,\tilde{\varphi}\rangle\varphi(x)}{4t},
\end{equation*}
which implies
\begin{equation*}
\mathtt{v}_s [{\theta f}/{t}] (x) \geq \frac{\langle k\wedge \theta f,\tilde{\varphi}\rangle\varphi(x)}{4t}.
\end{equation*}
Therefore, 
$$
\int_0^t {\mathtt{H}}[\sv_{s}[{\theta f}/{t} ]] \dd s \geq \int_0^{t_0} {\mathtt{H}}[\sv_{s}[{\theta f}/{t} ]]\dd s + \int_{t_0}^t {\mathtt{H}}\left[\frac{\langle k\wedge \theta f,\tilde{\varphi}\rangle\varphi}{4t}\right] \dd s.
$$
The first term on the right-hand side converges to 0 and the second tends to $\langle k\wedge \theta f,\tilde{\varphi}\rangle \mathtt{I} [\varphi]/{4}$ again by monotone convergence. This shows that ${\mathtt{I}}[\varphi]<\infty$ is also a necessary condition.

%\footnote{\azul Pedro: second part, calculation of the limit distribution}

\medskip

\noindent{\bf Step 2:}
Next, we compute the explicit limit distribution assuming $\mathtt{I}[\varphi]<\infty$.
Recalling \eqref{evol-eq-SPI}, our starting point is the expression
\begin{align*}
    \mathbb{E}_{\mu} \left[ \exp(-\langle {\theta}f, Y_t\rangle/{t})\right] &=  \mathbb{E}_{\mu} \left[ \exp(-\langle {\theta}f, X_t\rangle/{t})\right] \notag\\
    &\hspace{2cm}\times{\exp }\left(- \int_0^t \alpha \Tilde{\mathcal{E}}\left[1-\prod_{i=1}^{\tilde{N}} \left( 1-\su_s[\mathrm{e}^{-{\theta}f/{t}}](x_i)\right)\right] \dd s\right)
\end{align*}

   We are interested in the second term on the right hand side since, thanks to \eqref{vt/t},  $\mathbb{E}_{\mu} [ \mathrm{e}^{-\theta \langle f, X_t\rangle/t }]\to 1$ as $t\to \infty$. %\footnote{\color{blue} Doesn't this last thing right here need {\azul (H3)}, I would say it certainly does in the superprocess setting if not in the non-local branching process setting). But if you can prove $\sv_t[f/t]\to0$ without {\azul (H3)} as in earlier footnote, then you won't need {\azul (H3)}.}
   Let us fix $f\in B^+(E)$ and $\theta \in (0, \infty)$ and, recalling \eqref{eq:at},  we choose $t_0>0 $ large enough such that $0\leq a_s[\mathrm{e}^{-{\theta f}/{t}}]\varphi(x) \leq \langle 1-\mathrm{e}^{-{\theta f}/{t}},\Tilde{\varphi}\rangle \|\varphi\| \leq {\theta}\langle f,\Tilde{\varphi}\rangle \|\varphi\|/{t}\leq 1$ for all $t\geq t_0$ and $s\in [0,t]$. Then
   \begin{align}
       \int_0^t \alpha \Tilde{\mathcal{E}}&\left[1-\prod_{i=1}^{\tilde{N}} \left( 1-\su_s[\mathrm{e}^{-{\theta f}/{t}}](x_i)\right)\right] \dd s \notag\\
       &=  \int_{0}^t \alpha \Tilde{\mathcal{E}}\left[\prod_{i=1}^{\tilde{N}} \left( 1-a_s[\mathrm{e}^{-{\theta f}/{t}}]\varphi(x_i)\right)-\prod_{i=1}^{\tilde{N}} \left( 1-\su_s[\mathrm{e}^{-{\theta f}/{t}}](x_i)\right)\right] \dd s \notag\\
       &\quad+ \int_{0}^t \alpha \Tilde{\mathcal{E}}\left[1-\prod_{i=1}^{\tilde{N}} \left( 1-a_s[\mathrm{e}^{-{\theta f}/{t}}]\varphi(x_i)\right)\right] \dd s.
       \label{twoterms}
   \end{align}
   The first term on the right hand side of \eqref{twoterms}  converges to 0 as $t\to\infty$. 
   Indeed, using the deterministic inequality
   \begin{equation*}
       \left|\prod_{i=1}^n(1-z_i) - \prod_{i=1}^n(1-y_i) \right|\leq \sum_{i=1}^n |z_i-y_i|, \qquad z_i,y_i \in [0,1],\, n\in \N,
   \end{equation*}
   we see that an upper bound for said term is
   \begin{align*}
       \int_{0}^t \alpha \Tilde{\mathcal{E}}\left[\langle \left| \su_s[\mathrm{e}^{-{\theta f}/{t}}]-a_s [\mathrm{e}^{-{\theta f}/{t}}] \varphi\right|,\Tilde{\mathcal{Z}}\rangle\right] \dd s &\leq \alpha \Tilde{\mathcal{E}}\left[\langle \varphi,\Tilde{\mathcal{Z}}\rangle\right]\int_{0}^t \left\| \frac{\su_s[\mathrm{e}^{-{\theta f}/{t}}]}{\varphi}-a_s [\mathrm{e}^{-{\theta f}/{t}}] \right\| \dd s.
   \end{align*}
   The statement follows from
   \begin{align}
       \left\| \frac{\su_s[\mathrm{e}^{-{\theta f}/{t}}]}{\varphi}-a_s [\mathrm{e}^{-{\theta f}/{t}}] \right\| &\leq C \left({\rm e}^{-\varepsilon s}\|\tfrac{\theta}{t}f\| + \int_0^s {\rm e}^{-\varepsilon(s-r)}\|\sT_r[\tfrac{\theta}{t}f]\|^2 \d r \right) \nonumber\\
       &\leq C  \left({\rm e}^{-\varepsilon s}\frac{\theta}{t}\|f\| + \frac{\theta^2}{t^2}\|f\|^2 \frac{1}{\varepsilon} \right),
       \label{messyinequality}
   \end{align}
where we used Lemma \ref{lem: iter} and remind the reader that the constant $C>0$ can change its value in the second inequality.

%Now, for the third term, we will show that
Now, for the second term on the right-hand side of \eqref{twoterms}, we will show that
\begin{align}
   & \lim_{t\to\infty}\int_{0}^t \alpha \Tilde{\mathcal{E}}\left[1-\prod_{i=1}^{\tilde{N}} \left( 1-a_s[\mathrm{e}^{-{\theta f}/{t}}]\varphi(x_i)\right)\right] \dd s\notag\\
   &\hspace{2cm} = \frac{2{\mathtt I}[\varphi]}{\langle  \mathbb V[\varphi], \tilde\varphi \rangle} \log \left(1 + \frac{1}{2} \theta\langle f, \tilde\varphi\rangle \langle  \mathbb V[\varphi], \tilde\varphi \rangle\right) ,
    \label{needto1}
\end{align}
which will then conclude the proof. 

\medskip
\noindent{\bf Step 3:} In this step we will convert the desired limit \eqref{needto1} to an easier to handle   form. 
It is not hard to see that
\begin{align*}
    \Tilde{\mathcal{E}}&\left[1-\prod_{i=1}^{\tilde{N}} \left( 1-a_s[\mathrm{e}^{-{\theta f}/{t}}]\varphi(x_i)\right)\right] \\ 
    &= \Tilde{\mathcal{E}}\left[\sum_{i=1}^{\tilde{N}} a_s[\mathrm{e}^{-{\theta f}/{t}}]\varphi(x_i)\int_0^1 \prod_{{\substack{j=1\\ j\neq i}}}^{\tilde{N}}\left( 1-z a_s[\mathrm{e}^{-{\theta f}/{t}}]\varphi(x_j)\right)\dd z\right] 
\end{align*}
by Taylor's remainder theorem, so that
\begin{align*}
&\Tilde{\mathcal{E}}\left[1-\prod_{i=1}^{\tilde{N}} \left( 1-a_s[\mathrm{e}^{-{\theta f}/{t}}]\varphi(x_i)\right)\right]\notag\\
& = \Tilde{\mathcal{E}}\left[\langle a_s[\mathrm{e}^{-{\theta f}/{t}}]\varphi, \Tilde{\mathcal{Z}}\rangle\right] \notag\\ 
&\hspace{1cm}    +\Tilde{\mathcal{E}}\left[\sum_{i=1}^{\tilde{N}} a_s[\mathrm{e}^{-{\theta f}/{t}}]\varphi(x_i)\int_0^1 \left(\prod_{{\substack{j=1\\ j\neq i}}}^{\tilde{N}}\left( 1-z a_s[\mathrm{e}^{-{\theta f}/{t}}]\varphi(x_j)\right)-1\right)\dd z\right].
\end{align*}
Taking into account  the deterministic inequalities
   \begin{equation}\label{ineq:prod}
       0\leq 1-\prod_{i=1}^n(1-z_i)\leq \sum_{i=1}^n z_i, \qquad z_i \in [0,1],\, n\in \N
   \end{equation}
and $0\leq a_s[\exp(-{\theta f}/{t})]  \leq \langle 1-\exp(-{\theta f}/{t}),\Tilde{\varphi}\rangle   \leq {\theta}\langle f,\Tilde{\varphi}\rangle  /{t}$, we get
\begin{align*}
    \Tilde{\mathcal{E}}&\left[\sum_{i=1}^{\tilde{N}} a_s[\mathrm{e}^{-{\theta f}/{t}}]\varphi(x_i)\int_0^1 \left|\prod_{{\substack{j=1\\ j\neq i}}}^{\tilde{N}}\left( 1-z a_s[\mathrm{e}^{-{\theta f}/{t}}]\varphi(x_j)\right)-1\right|\dd z\right] \\
    &\leq \Tilde{\mathcal{E}}\left[\sum_{i=1}^{\tilde{N}} a_s[\mathrm{e}^{-{\theta f}/{t}}]\varphi(x_i)\sum_{{\substack{j=1\\ j\neq i}}}^{\tilde{N}}a_s[\mathrm{e}^{-{\theta f}/{t}}]\varphi(x_j) \mathbf{1}_{\{\langle \varphi, \Tilde{\mathcal{Z}}\rangle <\sqrt{t}\}}\right] + a_s[\mathrm{e}^{-{\theta f}/{t}}]\Tilde{\mathcal{E}}\left[\langle \varphi, \Tilde{\mathcal{Z}}\rangle \mathbf{1}_{\{\langle \varphi, \Tilde{\mathcal{Z}}\rangle \geq \sqrt{t}\}}\right] \\
    &\leq\frac{\theta^2}{t^2} \langle f, \Tilde{\varphi}\rangle^2\sqrt{t} \Tilde{\mathcal{E}}\left[\langle \varphi, \Tilde{\mathcal{Z}}\rangle \mathbf{1}_{\{\langle \varphi, \Tilde{\mathcal{Z}}\rangle <\sqrt{t}\}}\right] + \frac{\theta^2}{t}\langle f, \Tilde{\varphi}\rangle\Tilde{\mathcal{E}}\left[\langle \varphi, \Tilde{\mathcal{Z}}\rangle \mathbf{1}_{\{\langle \varphi, \Tilde{\mathcal{Z}}\rangle \geq \sqrt{t}\}}\right]
\end{align*}
Integrating the previous expression over $s\in[0,t]$, we see that it converges to $0$. Consequently,
\begin{align*}
    \lim_{t\to\infty}\int_{0}^t \alpha \Tilde{\mathcal{E}}\left[1-\prod_{i=1}^{\tilde{N}} \left( 1-a_s[\mathrm{e}^{-{\theta f}/{t}}]\varphi(x_i)\right)\right] \dd s = 
    \lim_{t\to\infty}\int_{0}^t \alpha \Tilde{\mathcal{E}}\left[\langle a_s[\mathrm{e}^{-{\theta f}/{t}}]\varphi, \Tilde{\mathcal{Z}},\rangle\right] \dd s,
\end{align*}
which means that in order to prove \eqref{needto1}, it is sufficient to show that
\begin{equation}\label{lim-int-a-BPI}
       \lim_{t\to \infty} \int_{0}^t a_s[\mathrm{e}^{-{\theta f}/{t}}] \dd s = \frac{2}{\langle  \mathbb V[\varphi], \tilde\varphi \rangle} \log \left(1 + \frac{1}{2} \theta\langle f, \tilde\varphi\rangle \langle  \mathbb V[\varphi], \tilde\varphi \rangle\right).
\end{equation}

   \medskip

\noindent{\bf Step 4:} To prove \eqref{lim-int-a-BPI}, 
let us momentarily fix $f\in B^+(E)$ such that $\inf_{x\in E} f(x) >0$.
From the definition of $a_s[\cdot]$ in \eqref{eq:at},  in a similar spirit to \eqref{modinverse} and \eqref{E:agtbounds BMPI}, we have 
   \begin{equation*}
       \frac{1}{a_s[\mathrm{e}^{-{\theta f}/{t}}]}-\frac{1}{a_0[\mathrm{e}^{-{\theta f}/{t}}]} %= \int_0^s \frac{\partial }{\partial r}\left(\frac{1}{a_r[\mathrm{e}^{-{\theta f}/{t}}]}\right)\dd r
       =   \int_0^s \frac{\langle \mathtt{A}[\mathtt{u}_r[\mathrm{e}^{-{\theta f}/{t}}]],\tilde{\varphi}\rangle}{a_r[\mathrm{e}^{-{\theta f}/{t}}]^2}\dd r
%       .
%   \end{equation*}
%Rewriting the right hand side as
%\begin{equation*}
%       \int_0^s \frac{\langle \mathtt{A}[\mathtt{u}_r[\mathrm{e}^{-{\theta f}/{t}}]],\tilde{\varphi}\rangle}{a_r[\mathrm{e}^{-{\theta f}/{t}}]^2}\dd r 
       = \langle \tfrac12 \mathbb{V}[\varphi],\tilde\varphi \rangle s + F_1[f](s,t) + F_2[f](s,t)
   \end{equation*}
and using $a_0[\mathrm{e}^{-{\theta f}/{t}}] = \langle 1- \mathrm{e}^{-{\theta f}/{t} },\tilde\varphi\rangle$
 we easily obtain
\begin{equation}\label{eq-a_s BMPI}
       a_s[\mathrm{e}^{-{\theta f}/{t}}] = \frac{1}{ \langle 1- \mathrm{e}^{-{\theta f}/{t} },\tilde\varphi\rangle^{-1} + \langle \tfrac12 \mathbb{V}[\varphi],\tilde\varphi \rangle s + F_1[f](s,t) + F_2[f](s,t)},
\end{equation}
where
   \begin{equation*}
       F_1[f](s,t) = \int_0^s \frac{\langle  \mathtt{A}[\mathtt{u}_r[\mathrm{e}^{-{\theta f}/{t}}]]-\frac{1}{2}\mathbb{V}[\mathtt{u}_r[\mathrm{e}^{-{\theta f}/{t}}]],\tilde{\varphi}\rangle}{a_r[\mathrm{e}^{-{\theta f}/{t}}]^2} \dd r,
   \end{equation*}
   \begin{equation*}
       F_2[f](s,t) = \int_0^s \frac{\frac{1}{2}\langle \mathbb{V}[\mathtt{u}_r[\mathrm{e}^{-{\theta f}/{t}}]]-\mathbb{V}[a_r[\mathrm{e}^{-{\theta f}/{t}}]\varphi],\tilde{\varphi}\rangle}{a_r[\mathrm{e}^{-{\theta f}/{t}}]^2} \dd r.
   \end{equation*}

Next, we look at controlling $F_1[f](s,t)$  and $F_2[f](s,t)$. In what follows, the constant $C$ may vary in value from line to line and formula to formula, however its value is not important other than it is strictly positive and finite.
By Lemma \ref{Vcont},
\begin{equation*}
       |F_2[f](s,t)| \leq \int_0^s \frac{1}{2}\langle \left\|\mathbb{V}[{\mathtt{u}_r[\mathrm{e}^{-{\theta f}/{t}}]}/{a_r[\mathrm{e}^{-{\theta f}/{t}}]}]-\mathbb{V}[\varphi]\right\|,\tilde{\varphi}\rangle \dd r \leq C  
     \int_0^s \left\|\frac{\su_r[\mathrm{e}^{-{\theta f}/{t}}]}{a_r[\mathrm{e}^{-{\theta f}/{t}}]} - \varphi \right\| \dd r.
\end{equation*}
From \cite{GHK}, we know that  there exists $C_1>0$ such that $\mathbb{E}^{\varphi}_{\delta_x}[\langle \varphi, X_r \rangle] \leq C_1(1+r)$ (see \cite{GHK}) for all $r\in [0,t]$ and
hence, from 
Lemma \ref{lem: a-lowerbd} 
\[
a_r[\mathrm{e}^{-{\theta f}/{t}}]\ge \frac{C}{C_1(1 +r)+ {\|\varphi\| t}({\theta \inf_{y\in E} f(y)})^{-1} } \ge \frac{C/t}{1+({\theta \inf_{y\in E} f(y)})^{-1} }, 
\]
for some $C>0$.
%where the first inequality follows from Lemma \ref{lem: a-lowerbd} and in the second one we use $t>1$, $r\leq t$ and the change of value of $C$. 
Combining this with \eqref{messyinequality},
%Lemma \ref{lem: iter},
%\begin{equation*}
%    \left\|\frac{\su_r[\mathrm{e}^{-{\theta f}/{t}}]}{\varphi} - a_r[\mathrm{e}^{-{\theta f}/{t}}]  \right\|  \leq  C  \left({\rm e}^{-\varepsilon r}\frac{\theta}{t}\|f\| + \frac{\theta^2}{t^2}\|f\|^2 \frac{1}{\varepsilon} \right).
%\end{equation*}
 we see that
%\begin{align*}
%    \left\|\frac{\sV_s[\frac{\theta}{t}f]}{a_s[\frac{\theta}{t}f]} - \varphi \right\| & \leq \frac{C_2 +C_3 s+ \frac{\|\varphi\|}{\theta \inf_{y\in E} f(y)} t}{C_1} C  \left({\rm e}^{-\varepsilon s}\|\frac{\theta}{t}f\| + (\frac{\theta}{t}C(f))^2 \frac{1-{\rm e}^{-\varepsilon s}}{\varepsilon} \right) \\
%\end{align*}
\begin{equation}\label{aux-step}
    \left\|\frac{\su_r[\mathrm{e}^{-{\theta f}/{t}}]}{a_r[\mathrm{e}^{-{\theta f}/{t}}]} - \varphi \right\|  \leq  C\left(1+\frac{1}{\theta \inf_{y\in E} f(y)}\right)  \left({\rm e}^{-\varepsilon r}\theta\|f\| + \frac{\theta^2}{t}\|f\|^2 \frac{1}{\varepsilon} \right) \\
\end{equation}
and hence
\begin{equation}
    \int_0^s \left\|\frac{\su_r[\mathrm{e}^{-{\theta f}/{t}}]}{a_r[\mathrm{e}^{-{\theta f}/{t}}]} - \varphi \right\| \dd r \leq  \frac{C}{\varepsilon}\|f\|\left(\theta +\frac{1}{\inf_{y\in E} f(y)}\right)  \left(1 + \theta\|f\| \right),
    \label{intmess}
\end{equation}
where the constant $C>0$ may change from line to line. 
%\[\int_0^s \left\|\frac{\sV_r[\frac{\theta}{t}f]}{a_r[\frac{\theta}{t}f]} - \varphi \right\| \dd r \leq \frac{C_2 +C_3 + \frac{\|\varphi\|}{\theta \inf_{y\in E} f(y)} }{C_1} C  \left(\frac{1}{\varepsilon}\|\theta f\| + \frac{s}{t}(\theta C(f))^2 \frac{1}{\varepsilon} \right)\] 
As a consequence we have the bound 
\begin{align}
    |F_2[f](s,t)|&\leq C_{F_2}, \label{bound-F2 BPI}
\end{align}
where $C_{F_2}>0$ is a  constant.

Applying Lemma \ref{GVsmall}, for each $\epsilon > 0$ there exists $t_0>1$ large enough such that
\begin{equation}
    |F_1[f](s,t)| \leq \int_0^s \frac{\langle \|\mathtt{A}[\mathtt{u}_r[\mathrm{e}^{-{\theta f}/{t}}]]-\frac{1}{2}\mathbb{V}[\mathtt{u}_r[\mathrm{e}^{-{\theta f}/{t}}]]\|,\tilde{\varphi}\rangle}{a_r[\mathrm{e}^{-{\theta f}/{t}}]^2} \dd r \leq  
    \epsilon \int_0^s  \left\|\frac{\su_r[\mathrm{e}^{-{\theta f}/{t}}]}{a_r[\mathrm{e}^{-{\theta f}/{t}}]} \right\|^2 \dd r 
    \label{F1est}
\end{equation}
for all $s\in [0,t]$ and $t>t_0$. Using the triangle inequality,
\[\left\|\frac{\su_r[\mathrm{e}^{-{\theta f}/{t}}]}{a_r[\mathrm{e}^{-{\theta f}/{t}}]} \right\|^2 \leq \left\|\frac{\su_r[\mathrm{e}^{-{\theta f}/{t}}]}{a_r[\mathrm{e}^{-{\theta f}/{t}}]} -\varphi\right\|^2 + \|\varphi\|^2\]
and hence, appealing to \eqref{intmess}, we have back in \eqref{F1est} that
\begin{align}
    |F_1[f](s,t)|&\leq \epsilon (C_{F_1} + s) \label{bound-F1 BPI}
\end{align}
for some constant $C_{F_1}>0$.

With \eqref{bound-F2 BPI} and \eqref{bound-F1 BPI} in hand,  the integral in \eqref{lim-int-a-BPI} is bounded below and above (respectively depending on the $\pm$ sign) by
\begin{align*}
   & \int_0^t \frac{\dd s}{\langle 1- \mathrm{e}^{-{\theta f}/{t} },\tilde\varphi\rangle^{-1} + \langle \tfrac12 \mathbb{V}[\varphi],\tilde\varphi \rangle s \pm \left(\epsilon (C_{F_1} + s) + C_{F_2} \right)}\notag\\
   & \hspace{2cm}= \frac{1}
   {\langle \tfrac12 \mathbb{V}[\varphi],\tilde\varphi \rangle \pm \epsilon}
   {\log \left(1 + \frac{\left(\langle \tfrac12 \mathbb{V}[\varphi],\tilde\varphi \rangle \pm \epsilon \right) t}{\langle 1- \mathrm{e}^{-{\theta f}/{t} },\tilde\varphi\rangle^{-1} \pm \left(\epsilon C_{F_1}  + C_{F_2} \right)}\right)}.
\end{align*}
Taking limit as $t \to \infty$, we get
\begin{equation*}
    \frac{\log \left(1 + \theta \langle f, \Tilde{\varphi}\rangle \left(\langle \tfrac12 \mathbb{V}[\varphi],\tilde\varphi \rangle \pm \epsilon \right)\right)}{\langle \tfrac12 \mathbb{V}[\varphi],\tilde\varphi \rangle \pm \epsilon}.
\end{equation*}
Since $\epsilon$ is as small as we want, we achieve the desired limit as in \eqref{lim-int-a-BPI}.

\medskip

\noindent{\bf Step 5:}
To complete the proof of \eqref{lim-int-a-BPI},  we want to remove the assumption that $f\in B^+(E)$ satisfies $\inf_{x\in E} f(x) >0$.
To do this, we define
  \begin{equation*}
      \sF[h] := \lim_{t\to \infty} \int_{0}^t a_s[\mathrm{e}^{-{h}/{t}}] \dd s, \qquad h\in B^+(E).
  \end{equation*}
  {The functional $\mathtt{F}$ is well defined because the integrand of the above integral is non-negative and $\int_{0}^t a_s[\mathrm{e}^{-{h}/{t}}] \dd s \leq \int_{0}^t \langle 1 - \mathrm{e}^{-{h}/{t}},\tilde\varphi\rangle \dd s \leq \langle h ,\tilde\varphi\rangle .$}
  We claim that  $\sF$ is a continuous functional. %, and from this, we can drop the the current requirement  that $\inf_{x\in E} f(x) >0$ in our proof of  \eqref{lim-int-a-BPI}.
  Indeed,
given $h,g\in B^+(E)$, and taking into account the deterministic inequality $|\mathrm{e}^{-x}-\mathrm{e}^{-y}|\leq |x-y|$ for all $x,y\geq 0$, the claim follows from the fact that
\begin{align*}
    |\sF[h] - \sF[g]| &\leq \lim_{t\to\infty} \int_{0}^t \langle \mathbb{E}_{\delta_\cdot} \left[ |\mathrm{e}^{-\langle {h},X_s\rangle/{t} }-\mathrm{e}^{-\langle {g},X_s\rangle /{t}}|\right],\tilde\varphi\rangle \dd s \\
    &\leq \lim_{t\to\infty}\frac{1}{t} \int_{0}^t \langle \mathtt{T}_s[{\|h-g\|}],\tilde\varphi\rangle \dd s = \langle 1,  \Tilde{\varphi}\rangle \|h-g\|.
\end{align*}

{Now we show} that   \eqref{lim-int-a-BPI} holds for any function $f\in B^+(E)$ {satisfying $\inf_{x\in E} f(x)  = 0$}. By considering a  sequence of functions $f_n=\max (n^{-1},f)$, $n\in\N$, noting that  $\inf_{x\in E} f_n(x)>0$,  $\lim_{n\to\infty} \|f_n-f\| = 0$ and 
\begin{equation*}
      \sF[\theta f_n] = \frac{2}{\langle  \mathbb V[\varphi], \tilde\varphi \rangle} \log \left(1 + \frac{1}{2} \theta\langle f_n, \tilde\varphi\rangle \langle  \mathbb V[\varphi], \tilde\varphi \rangle\right), \qquad n\in\N,
  \end{equation*}
 continuity gives us that $\sF[\theta f] = \lim_{n\to\infty} \sF[\theta f_n]$ and that  \eqref{lim-int-a-BPI} holds for $f\in B^+(E)$ with $\inf_{x\in E} f(x) =0$.

\bigskip

To conclude Steps 2-5, Step 5 ensures that \eqref{lim-int-a-BPI} holds, which ensures that \eqref{needto1}  holds, which, in turn, from \eqref{twoterms} gives us the limiting result
\begin{align*}
    \mathbb{E}_{\mu} \left[ \exp(- {\theta}\langle f, Y_t\rangle/{t})\right]
    = \left(1 +  \theta\frac{1}{2}\langle f, \tilde\varphi\rangle \langle  \mathbb V[\varphi], \tilde\varphi \rangle\right)^{-{2{\mathtt I}[\varphi]}/{\langle  \mathbb V[\varphi], \tilde\varphi \rangle}}
    \end{align*}
    as required.\hfill$\square$

\section{Proof of Theorem \ref{subcrit}}%\footnote{\azul Pedro: the proofs are exactly the same for BMPI or SPI}
% The reader should note that Theorem \ref{subcrit} does not assume hypotheses {\azul (H3)} and {\azul (H4)}. The reason that  they are not necessary is because Kolmogorov's estimate is not involved in the proof of this result. Instead, we will use Jensen's inequality.
% We give the proof of Theorem \ref{subcrit} only in the setting of BMPIs. The proof in the context of superprocesses is exactly the same with $\mathtt{V}_t[f]$ and $\chi$ playing the role of $\mathtt{v}_t[f]$ and $\mathtt{H}$. We leave those details to the reader.
As usual, we restrict ourselves to the setting of non-local branching Markov processes, noting that  the non-local superprocess setting is almost verbatim {(with $\sV_t$ playing the role of $\sv_t$)}, with a full proof available in \cite{thesis}.

%\subsection{Non-local branching Markov processes with immigration}

From equations \eqref{eq: semigroup BPI} and \eqref{evol-eq-BMPI}, it is clear that $Y_t \to Y_\infty$ weakly as $t\to\infty$ if, and only if, for all $f\in B^+(E)$, we have $\lim_{t\to\infty} \|\mathtt{v}_t[f]\|<\infty $ and
\begin{equation}\label{cond-stat-dist-BPI}
    \int_0^\infty \mathtt{H}[\mathtt{v}_{s}[f]] \mathrm{d}s < \infty.
\end{equation}
%{\color{blue}where we have used {\azul (H3)} to ensure that $\sv_t[f]\to0$ as $t\to\infty$.}
%In particular, we will see below that $\lim_{t\to\infty} \|\mathtt{v}_t[f]\| = 0$ and, consequently, $Y_\infty$ is given by  $   \mathbb{E}\left[\mathrm{e}^{-\langle f,Y_\infty\rangle}\right] = \exp \left(-\int_0^\infty \mathtt{H}[\mathtt{v}_{s}[f]] \mathrm{d}s\right)$. 

The strategy now is to prove the equivalence between \eqref{cond-stat-dist-BPI} and the integral  condition,
\begin{equation}\label{integral-test-subcrit}
    \int_0^{z_0} \frac{\mathtt{H}[z\varphi]}{z}\mathrm{d}z <\infty \text{ for some } z_0>0.
\end{equation}
We proceed by finding two functions, one that bounds $\mathtt{v}_t[\cdot]$ above and one below. The necessesity and sufficiency   of \eqref{integral-test-subcrit} follows, respectively, from the monotonicity of the immigration mechanism $\mathtt{H}[\cdot]$ applied to said bounding functions.

Let us start by showing that \eqref{integral-test-subcrit} is a sufficient condition. 
%By Jensen's inequality,
%\[\mathrm{e}^{-\mathtt{v}_t[f](x)} \geq \mathrm{e}^{-\mathtt{T}_t[f](x)},  \]
%for all $t\geq 0$, $f\in B^+(E)$ and $x\in E$. 
For all $t\geq 0$, $f\in B^+(E)$ and $x\in E$, we have
\begin{equation}\label{upper-bound-expectation-semi}
    \mathtt{v}_t[f](x) \leq \mathtt{T}_t[f](x)  \leq \left(\langle 1,\Tilde{\varphi}\rangle +\Delta\right)\|f\|\mathrm{e}^{\lambda t}\varphi(x),
\end{equation}
where the first inequality is due to Jensen's inequality and the second is thanks to {(H2)}.
As a consequence, $\lim_{t\to\infty} \|\mathtt{v}_t[f]\| = 0$ because $\lambda  <0$ by assumption. Moreover, with the change of variables $z = \left(\langle 1,\Tilde{\varphi}\rangle +\Delta\right)\|f\|\mathrm{e}^{\lambda t}$, we get
\[
\int_{0}^\infty \mathtt{H}[\mathtt{v}_t[f]]\mathrm{d}t \leq -\frac{1}{\lambda}\int_0^{\left(\langle 1,\Tilde{\varphi}\rangle +\Delta\right)\|f\|} \frac{\mathtt{H}[z\varphi]}{ z }\mathrm{d}z,
\]
that is, \eqref{integral-test-subcrit}  is a sufficient condition.

Let us now show that \eqref{integral-test-subcrit} is necessary. The counterpart inequalities of \eqref{ineq-moment-1} and \eqref{ineq-moment-2} at subcriticality (see \cite{GHKcorr}) are
\begin{align}
    \mathtt{T}_t [g] (x) &\geq \left(\langle g,\tilde{\varphi}\rangle  - \|g\|\Delta_t\right) \mathrm{e}^{\lambda t}\varphi(x), \label{ineq-moment-1-subcrit}\\
    \mathtt{T}_t^{(2)} [g] (x) &\leq \left( L_2 (g) + \|g\|^2\Delta_t^{(2)}\right)\mathrm{e}^{\lambda t} \varphi(x),\label{ineq-moment-2-subcrit}
\end{align}
for all $t\geq 0$, $g\in B^+(E)$ and $x\in E$, where $\mathtt{T}_s^{(2)} [g] (x) = \mathbb{E}_{\delta_x} [\langle g , X_s \rangle^2]$, $\Delta_t$ is defined in {(H2)} and
$\Delta_t^{(2)} = \sup_{x\in E, \, f\in B^+_1(E)} |\mathrm{e}^{-\lambda t}\varphi(x)^{-1}\sT_t^{(2)}\bra{f}(x)-L_2(f)|$ with 
\begin{equation}\label{def-L_2}
    L_2 (g) = \langle g^2, \Tilde{\varphi}\rangle+\int_0^\infty \mathrm{e}^{-\lambda s}\langle \mathbb{V}[\mathtt{T}_s[g]],\Tilde{\varphi}\rangle \mathrm{d}s.
\end{equation}
Using the upper bound  for the expectation semigroup in \eqref{upper-bound-expectation-semi}, we  obtain $L_2(g) \leq [\langle 1, \Tilde{\varphi}\rangle - \lambda^{-1}(\langle 1, \Tilde{\varphi}\rangle+\Delta)^2 \langle \mathbb{V}[\varphi],\Tilde{\varphi}\rangle]\|g\|^2$. That is, there exists a constant $C>0$ such that 
\begin{equation}\label{upper-bound-L_2}
    L_2 (g) \leq C \|g\|^2, \qquad \text{for all } g\in B^+(E).
\end{equation}
%Just a quick observation, in the context of superprocesses, the definition of $L_2(g)$ is simply the second term on the right hand side of \eqref{def-L_2}. This small change does not affect the inequality \eqref{upper-bound-L_2} that continues to be satisfied. 

Combining \eqref{ineq_non-lin_moment-2}, \eqref{ineq-moment-1-subcrit}, \eqref{ineq-moment-2-subcrit} and \eqref{upper-bound-L_2}, we have
$$
\mathrm{e}^{-\mathtt{v}_t[g](x)} \leq 1 - \left(\langle g,\Tilde{\varphi}\rangle-\|g\|\Delta_t\right)\mathrm{e}^{\lambda t}\varphi(x) + \frac{1}{2}\left(C + \Delta^{(2)}\right) \|g\|^2\mathrm{e}^{\lambda t}\varphi(x),
$$
where  $\Delta^{(2)} = \sup_{t\geq 0} \Delta_t^{(2)}$, which is finite under {(H1)}; see \cite{GHKcorr}. Let us apply this expression to $g = \mathrm{e}^{\lambda t}f\leq f$, with $f\in B^+(E)$, then
$$
\mathrm{e}^{-\mathtt{v}_t[f](x)} \leq \mathrm{e}^{-\mathtt{v}_t[\mathrm{e}^{\lambda t}f](x)} \leq 1 - \left(\langle f,\Tilde{\varphi}\rangle-\|f\|\Delta_t\right)\mathrm{e}^{2\lambda t}\varphi(x) + \frac{C + \Delta^{(2)}}{2} \|f\|^2\mathrm{e}^{3\lambda t}\varphi(x).
$$
As $\Delta_t\to0$ as $t\to\infty$, we know that for each $f\in B^+(E)$ there exists $t_0>0$ large enough such that for each $t\geq t_0$, 
\[
\|f\|\Delta_t \leq \frac{1}{3}\langle f,\Tilde{\varphi}\rangle \qquad \text{and} \qquad \frac{C + \Delta^{(2)}}{2} \|f\|^2\mathrm{e}^{3\lambda t} \leq \frac{1}{3}\langle f,\Tilde{\varphi}\rangle\mathrm{e}^{2\lambda t}.
\]
Finally, for all $t\geq t_0$, we achieve our desired lower bound, 
$$
\mathrm{e}^{-\mathtt{v}_t[f](x)} \leq 1 - \frac{1}{3}\langle f,\Tilde{\varphi}\rangle\mathrm{e}^{2\lambda t}\varphi(x) \quad \Longrightarrow \quad \mathtt{v}_t[f](x) \geq \frac{1}{3}\langle f,\Tilde{\varphi}\rangle\mathrm{e}^{2\lambda t}\varphi(x).
$$
Under the change of variables  $z = \frac{1}{3}\langle f, \Tilde{\varphi}\rangle \mathrm{e}^{2\lambda t}$, we see that \eqref{integral-test-subcrit} is a necessary condition for \eqref{cond-stat-dist-BPI} due to the fact that
\[
\int_{t_0}^\infty \mathtt{H}[\mathtt{v}_t[f]]\mathrm{d}t \geq -\frac{1}{2\lambda}\int_0^{\frac{1}{3}\langle f,\Tilde{\varphi}\rangle\mathrm{e}^{2\lambda t_0}} \frac{\mathtt{H}[z\varphi]}{ z }\mathrm{d}z.
\]
The last step is the equivalence between \eqref{integral-test-subcrit} and the log moment condition \eqref{log-immig-BMPI}. This follows by  Tonelli's theorem,
\begin{equation*}
    \int_0^{z_0} \frac{\mathtt{H}[z\varphi]}{ z } \mathrm{d}z = \alpha \Tilde{\mathcal{E}} \left[\int_0^{z_0 \langle\varphi, \tilde{\mathcal Z}\rangle} \frac{1-\mathrm{e}^{-y}}{y} \mathrm{d}y\right]< \infty,
\end{equation*}
for some $z_0>0$. The latter holds iff $\Tilde{\mathcal{E}}[\log  (1 + \langle\varphi, \tilde{\mathcal Z}\rangle)]<\infty$. \qed

\begin{remark}\rm
Roughly speaking, paraphrasing Corollary 9.53 in \cite{ZL11} in the setting of non-local Markov branching processes, it states that, for $\sv_t[1](x)$,  uniform (in $x\in E$) exponential bounds (in time) are needed   to ensure that $\Tilde{\mathcal{E}}[\log  (1 + \langle\varphi, \tilde{\mathcal Z}\rangle)]<\infty$ is a necessary and sufficient condition for a stationary distribution to exist. 
We note that the proof above   essentially verifies such exponential temporal bounds {starting at a certain instant $t_0>0$}, albeit they are not uniform in $x\in E$. 
\end{remark}

\section*{Acknowledgments} AEK and EH acknowledge the support of EPSRC programme grant EP/W026899/2. EP acknowledges the support of UKRI Future Leaders Fellowship MR/W008513/1.  VR is grateful for additional financial support  from  CONAHCyT-Mexico, grant nr. 852367. PM-C acknowledges grants FPU20/06588 and EST23/00817, funded by the Spanish Ministry of Universities, and grant PID2019-108211GB-I00, funded by MCIN/AEI/10.13039/501100011033. Part  of this work was undertaken whilst PM-C and VR were visiting the University of Warwick for an extended period. They would like to thank their hosts for partial financial support and for their kindness and hospitality.

\bibliography{references}{}
\bibliographystyle{abbrv}

\end{document}